\documentclass[11pt,reqno]{amsart}
\makeatletter
\def\subsection{\@startsection{subsection}{3}\z@{.5\linespacing\@plus.7\linespacing}{.7\linespacing}{\normalfont\bfseries}}
\makeatother

\usepackage{amsmath,amsfonts,amscd,amssymb}
\usepackage{amsmath, amsthm, amsopn, amssymb, graphicx, enumerate, geometry, afterpage, caption, subcaption, color, textcomp, bbm,bm}
\usepackage{stmaryrd}
\usepackage[bbgreekl]{mathbbol}
\usepackage[utf8]{inputenc}
\usepackage[T1]{fontenc}
\usepackage{xspace}
\usepackage{bbold}

\usepackage{amsmath,amsfonts,amsthm,amsbsy,amssymb,dsfont,stmaryrd}
\usepackage{subcaption, bbm}
\usepackage[dvipsnames]{xcolor}
\usepackage{braket}

\date{1st version: October 22, 2021; revised May 27, 2022}
\usepackage{mathtools}

\numberwithin{equation}{section}

\newcommand\myshade{85}
\colorlet{mylinkcolor}{violet}
\colorlet{mycitecolor}{YellowOrange}
\colorlet{myurlcolor}{Aquamarine}

\usepackage[unicode=true,pdfusetitle,
bookmarks=true,bookmarksnumbered=false,
breaklinks=false,pdfborder={0 0 1},backref=false,
linkcolor  = mylinkcolor!\myshade!black,
citecolor  = mycitecolor!\myshade!black,
urlcolor   = myurlcolor!\myshade!black,
colorlinks = true,ocgcolorlinks]{hyperref}
\usepackage[nameinlink]{cleveref}

\usepackage{tikz}
\usetikzlibrary{arrows}
\usetikzlibrary{intersections}
\usetikzlibrary{decorations.pathmorphing}
\usetikzlibrary{shapes.misc}
\tikzset{cross/.style={cross out, draw=black, minimum size=2*(#1-\pgflinewidth), inner sep=0pt, outer sep=0pt},
	cross/.default={1pt}}
\usepackage[outline]{contour}
\contourlength{1.2pt}
\usetikzlibrary{calc}
\usepackage{graphicx}
\usepackage{caption}
\makeatletter
\def\parsept#1#2#3{%
	\def\nospace##1{\zap@space##1 \@empty}%
	\def\rawparsept(##1,##2){%
		\edef#1{\nospace{##1}}%
		\edef#2{\nospace{##2}}%
	}%
	\expandafter\rawparsept#3%
}
\makeatother

\definecolor{ct_black}{HTML}{000000}
\definecolor{ct_orange}{HTML}{ED872D}
\definecolor{ct_purple}{HTML}{7A68A6}
\definecolor{ct_blue}{HTML}{348ABD}
\definecolor{ct_turquoise}{HTML}{188487}
\definecolor{ct_red}{HTML}{E32636}
\definecolor{ct_pink}{HTML}{CF4457}
\definecolor{ct_green}{HTML}{467821}

\definecolor{ct2_green}{HTML}{9FF781}
\definecolor{ct2_green_dark}{HTML}{088A08}

\theoremstyle{plain}
\newtheorem{thm}{\protect\theoremname}[section]
\theoremstyle{plain}
\newtheorem{lem}[thm]{\protect\lemmaname}
\theoremstyle{plain}
\newtheorem{cor}[thm]{\protect\corollaryname}
\theoremstyle{plain}
\newtheorem{prop}[thm]{\protect\propositionname}
\newtheorem*{theorem-non}{Proposition}

\theoremstyle{remark}

\theoremstyle{plain}

\theoremstyle{remark}
\newtheorem{rem}[thm]{\protect\remarkname}           
\theoremstyle{definition}
\newtheorem{defn}[thm]{\protect\definitionname}       
\theoremstyle{plain}

\providecommand{\assumptionname}{Assumption}
\providecommand{\claimname}{Claim}
\providecommand{\corollaryname}{Corollary}
\providecommand{\definitionname}{Definition}
\providecommand{\lemmaname}{Lemma}
\providecommand{\propositionname}{Proposition}
\providecommand{\remarkname}{Remark}
\providecommand{\theoremname}{Theorem}
\providecommand{\examplename}{Example}
\providecommand{\conjecturename}{Conjecture}
\providecommand{\notationname}{Notation}

\crefname{section}{Section}{Sections}
\crefname{appendix}{Appendix}{Appendices}
\crefname{figure}{Figure}{Figures}
\crefname{subfigure}{Figure}{Figures}
\crefname{assumption}{Assumption}{Assumptions}
\crefname{thm}{Theorem}{Theorems}
\crefname{prop}{Proposition}{Propositions}
\crefname{lem}{Lemma}{Lemmas}
\crefformat{equation}{(#2#1#3)}

\crefrangelabelformat{equation}{(#3#1#4--#5#2#6)}

\crefmultiformat{equation}{(#2#1#3}{, #2#1#3)}{#2#1#3}{#2#1#3}
\Crefmultiformat{equation}{(#2#1#3}{, #2#1#3)}{#2#1#3}{#2#1#3}

\def\be{\begin{equation}}
	\def\ee{\end{equation}}
\def\ba{\begin{align}}
	\def\ea{\end{align}}

\newcommand{\E}{\mathbb{E}}
\newcommand{\Z}{ {\mathbb Z} }
\newcommand{\ZGF}{ {\mathbb Z} \mathrm{GF}}
\newcommand{\ZUF}{ {\mathbb Z} U \mathrm{F}}

\newcommand{\ZBF}{ {\mathbb Z} \mathrm{BF}}

\def\G{{\mathcal G}}
\newcommand{\Int}{\mathrm{Int}}

\newcommand{\R}{\mathbb{R}}

\newtheorem*{lem*}{\protect\lemmaname}

\newcommand{\Oe}{\operatorname{e}}
\newcommand{\ii}{\operatorname{i}}
\newcommand{\Mat}{\operatorname{Mat}}
\newcommand{\ZZ}{\mathbb{Z}}

\newcommand{\re}{\operatorname{e}}

\newcommand{\NN}{\mathbb{N}}
\newcommand{\MM}{\mathbb{M}}
\newcommand{\RR}{\mathbb{R}}
\newcommand{\CC}{\mathbb{C}}
\newcommand{\PP}{\mathbb{P}}
\newcommand{\EE}{\mathbb{E}}

\newcommand{\calB}{\mathcal{B}}
\newcommand{\calM}{\mathcal{M}}
\newcommand{\calL}{\mathcal{L}}    
\newcommand{\calN}{\mathcal{N}}
\newcommand{\calF}{\mathcal{F}}

\newcommand{\calH}{\mathcal{H}}

\newcommand{\calG}{\mathcal{G}}
\newcommand{\calV}{\mathcal{V}}

\newcommand{\calE}{\mathcal{E}}

\newcommand{\1}{\mathbbm{1}}

\newcommand{\norm}[1]{\left\|#1\right\|}
\newcommand{\ip}[2]{\langle #1, #2 \rangle}

\newcommand{\dif}{\mathrm{d}}
\newcommand{\tr}{\operatorname{tr}}

\newcommand{\ve}{\varepsilon}
\newcommand{\vf}{\varphi}
\newcommand{\Id}{\mathds{1}}
\newcommand{\HH}{\mathbb{H}}

\newcommand{\sgn}{\operatorname{sgn}}

\DeclarePairedDelimiter\ceil{\lceil}{\rceil}

\newcommand{\Vill}{\mathrm{Vill}}

\usepackage{environ}

\NewEnviron{malign}{%
	\begin{align}\begin{split}
			\BODY
	\end{split}\end{align}
}

\def\deloc/{depinning}
\def\Deloc/{Depinning}
\def\DELOC/{DEPINNING} 
\def\ZGFs/{$\ZGF$}
\def\ZGFl/{integer-restricted Gaussian field} 

\title[The BKT phase and depinning in the $\Z$-valued Gaussian Field]  
{Depinning in integer-restricted Gaussian Fields \\
and  BKT phases of two-component spin models}         

\dedicatory{\hspace{6cm} \it  To Barry Simon, on the occasion of his $75^{\rm{th}}$ birthday
\vspace{-0.5cm}}

\author{Michael Aizenman}
\address{MA: Departments of Physics and  Mathematics, Princeton University, USA \hspace{2cm} \hfill
\\   \mbox{ } \quad
Weston Visiting Professor at the Weizmann Institute of Science, Israel}

\author{Matan Harel}
\address{MH: Department of Mathematics, Northeastern Univ., Boston, USA}

\author{Ron Peled}
\address{RP: School of Mathematical Sciences, Tel Aviv University, Israel}

\author{Jacob Shapiro}
\address{JS: Physics Department, Princeton University, Princeton, NJ 08544, USA}

\begin{document}
	
\begin{abstract}  For a family of integer-valued height functions defined over the faces of planar graphs, we establish a relation between the probability of connection by level sets and the spin-spin correlations of the dual $O(2)$ symmetric spin models formulated over the graphs' vertices.
The relation is used to show that in two dimensions the Villain spin model exhibits non-summable decay of correlations at any temperature at which the dual integer-restricted Gaussian field exhibits depinning.   For the latter, we devise a new monotonicity argument through which the  recent alternative proof by Lammers
of the existence of a depinning transition in two-dimensional graphs of degree three, is extended to all doubly-periodic graphs, in particular to $\mathbb{Z}^2$.
Essential use is made of the inequality of Regev and Stephens-Davidowitz, which allows also an alternative (to absolute-value FKG) proof of convergence of the height-function's distribution in the infinite-volume limit.
Similar results are established for the $XY$ spin model and its dual Bessel random height function.
Taken together these  statements yield a new perspective on the Berezinskii-Kosterlitz-Thouless phase transition in $O(2)$ spin models, and complete a new proof of depinning in two-dimensional integer-valued height functions.
\end{abstract}

	\maketitle

\vspace{-1.5cm}
\tableofcontents
\section{Introduction}
\label[Introduction]{sec:Introduction}
\subsection{An outline}

The subject of this paper is a pair of phenomena, each of which is special to  models formulated over $d=2$  dimensional graphs,  that are linked through a duality relation.  One is the depinning transition in random height functions, such as the integer-restricted Gaussian field ($\ZGF$), and the other is the Berezinskii-Kosterlitz-Thouless (BKT) phase with slowly decaying correlations in $O(2)$-invariant two-component spin systems, such as the Villain and the XY spin models.
Each of  these has been previously analyzed through its relation to a lattice system of integer charges with the two-dimensional Coulomb interaction (of logarithmic strength).
The results presented here employ probabilistic tools which yield another perspective on the existence of the two phenomena, and the link between the two.

To present the results in their simplest context first, we start by considering the dual pair of the Villain two-component spin model on $\Z^2$
and the integer-restricted Gaussian field on this graph's dual, which is also $\Z^2$ in this case.
However, we also explain how the analysis extends to other doubly-periodic graphs, and to other combinations of interest. Those include the $XY$ spin model and its dual, which is an integer-valued field whose Gibbs factor involves a modified Bessel function.

Beyond the example of $\Z^2$, our discussion applies to connected graphs $\G $ which are {\em planar} and {\em doubly-periodic} --  in the sense that there exists an embedding of $\G$ in $\mathbb{R}^2$ so that the natural action of $\mathbb{Z}^2$ by translation is a $\G$-automorphism.
Furthermore, it will be assumed that under the above embedding the graph's vertex set has no accumulation point.  We call graphs with the last property {\em tame}.

Among the graphs meeting these conditions, one finds the standard square, triangular, and hexagonal lattices, and also graphs  which are only quasi-transitive (i.e. of a finite periodicity cell). One may observe that if $\G$ is doubly-periodic and tame, then so is its dual graph $\G^*$.

We denote by $d_{\G}$ the graph distance, by $\Lambda_{v,r}$ the set of vertices $u$ for which $d_{\G}(v,u) \leq r$, and also write $\Lambda_L\equiv\Lambda_{v_0,L}$ with $v_0$ a fixed vertex to which we refer as the origin.

Following are the two leading examples of models and questions we study, and the main results in that context.

\subsection{The $\ZGF$ and its depinning transition}

The \ZGFl/ (\ZGFs/) over a locally finite and connected graph $\G$ of vertex set $\calV$ and edge set $\calE$
has, as its basic variables, the values of the random function $n:\calV\to\Z$.  The $\ZGF$ finite-volume partition function  in $ \calV_L\equiv \calV \cap [-L,L]^2$, which is taken here with the Dirichlet boundary conditions,  is
\begin{align} \label{ZGF_partition}
Z_{\lambda,L}^{\ZGF,\calG} \equiv \sum_{n:\calV _L \to\Z:\left.n\right|_{\partial \calV_L} = 0}
\exp\left(-\frac{\lambda }{2}
\sum_{\{u,v\}\in \calE_L} \left(n_u - n_v\right)^2 \right).
\end{align}
At any given $\lambda$, the probability of $n$ is given by the configuration's normalized contribution to the above sum.

Quite generally, \emph{Gaussian domination} holds in the sense that the fluctuations of the $\ZGF$ are upper bounded by those of the corresponding Gaussian free field (GFF)~\cite{FroPar78} (cf.~\cite{kharash2017fr}).

In two dimensions, such models exhibit a \emph{localization--delocalization}, or \emph{de-pinning}, phase transition, which can also be presented in terms of \emph{symmetry breaking}.

\begin{defn}\label{def:ZGF pinning}
		The $\ZGF$ is said to be \emph{pinned}, or exhibit \emph{localization}, at a specified $\lambda$ iff for every $x\in\calV$,
\be \label{eq_pinning}
		  \lim_{L\to\infty } \EE^{\ZGF,\calG}_{\lambda,L}\left[ n_x^2\right] < \infty \,.
\ee
The model is said to \emph{depin}, or be \emph{delocalized}, if  in the limit $L\to \infty$ the height variance diverges and, furthermore,
the surface  fluctuates away, with
\be\label{deepen}
  \lim_{L\to \infty} \PP^{\ZGF,\calG}_{\lambda,L}(|n_x|\le t)= 0
  \ee
 for each $x\in \mathcal V$ and $t<\infty$.

	\end{defn}
We postpone to Section 2 the more detailed presentation of the known results (and references) on the existence of the limit in \eqref{eq_pinning},  the complementarity of the above two conditions, and the limiting states'  basic properties.  Here, let us just note that
in the pinned phase, the system admits an infinite collection of infinite-volume Gibbs equilibrium states which are limits of the finite states with boundary values set at a common value $k\in\Z$, and hence differ by simple shifts. On doubly-periodic and tame graphs, which is the case of interest here, the limiting states will be ergodic.

The realizability of both phases at different values of $\lambda$ is limited to  two dimensions:
in one dimension, random height functions exhibit Brownian-bridge-type fluctuations, whereas in $d>2$ dimensions, the $\ZGF$ is pinned at all temperatures.

For $d=2$, localization at large enough $\lambda$ is not hard to prove using a Peierls-type argument.  Delocalization has been viewed as the more difficult phase to establish.
The occurrence of delocalization for the $\ZGF$ on $\Z^2$ and other doubly-periodic graphs was first established in \cite{FroSpe81} using the dual Coulomb gas perspective.  Recently, a new proof was presented by Lammers~\cite{Lam21} for doubly-periodic planar graphs of maximum degree $3$.
The proof is based on an \emph{influence-percolation} argument, for which the limitation on the graph's degree plays an important role.  The removal of this limitation is among the first results presented here.

The following summarizes the results presented here concerning the de-pinning phase transition per-se.

\begin{thm}\label{thm:deloc on Z^2}
For any locally finite and connected graph, there exists $\lambda_c(\G) \in [0, \infty]$ such that
the $\ZGF$ is pinned for $\lambda>\lambda_c(\G)$ and depinned for $\lambda<\lambda_c(\G)$. For planar, doubly-periodic graphs with those properties,
 $0 <  \lambda_c(\G) < \infty$.  In the case of
the $\ZGF$ model on $\ZZ^2$ at $\lambda =  \frac{2}{3}\ln2$ is delocalized (and by implication $\lambda_c(\Z^2) \geq \frac{2}{3}\ln2$).
\end{thm}

In the proof, we extend Lammers' result's applicability using a series of surgeries, through which the model's graph is changed into one of smaller degrees; this procedure is accompanied by controlled increases in the couplings. The transformations are rather natural, but the key point in the proof is that, in each step of the process, the height variance can only \emph{decrease}.
This is done through a pair of relatively recent inequalities for Gaussian measures on lattices, due to Regev and Stephens-Davidowitz \cite{RegDav17}. In our case, the lattice in question will be the $\ZGF$ configuration space.
These inequalities also imply that  the height variance is monotone in the model's coupling strength $\lambda$ and thus, as expected,  for each graph the depinning transition occurs at a single transitional value $\lambda_c$.
The stated bound on $\lambda_c$  builds on Lammers' proof of delocalization on cubic graphs at $\lambda \leq 2 \ln 2$.

\subsection{From  depinning in \ZGFs/ to  slow decay of correlations in spin models}

The spin models discussed here are systems of two component spins associated with the vertices of a graph.  Their configurations form a random $\mathbb{S}^1$-valued function $\sigma:  \mathcal V \to \mathbb{S}^1$,
with nearest-neighbor interaction and $O(2)$ rotational symmetry.  Their BKT  phase transition is  unique to two dimensions.

The example which we discuss first is the Villain model, with spins
$\sigma_x=\mathrm{e}^{i\theta_x}$,  whose partition function is
\be \label{Villain_Z1}
		Z^{\Vill,\G}_{\beta} =   \int_{\theta\in [-\pi,\pi)^{\calV}  }
		\left\{
		\prod_{\{u,v\}\in\calE} \left[ \sum_{m_{uv}\in\ZZ } \re^{-\frac{\beta}{2}   \left(\theta_u-\theta_v +2\pi m_{uv}\right)^2} \right]
		   \, \right\}
		\prod_{x\in \mathcal V} 	
		\dif{ \theta_x }\, ,
\ee
The corresponding Gibbs equilibrium state is given by the probability distribution
that is obtained by normalizing the measures integrated in \cref{Villain_Z1}.

This model can be given a probabilistically appealing
interpretation, with the Gibbs measure  presented
as the marginal distribution, restricted to the vertices,  of a more extended random spin function defined over the graph's edges.  The distribution of the extended model is formed by linking Wiener processes, each describing Brownian motion on the circle, tied at the vertices through a continuity constraint.
This formulations is used below to derive a  correlation inequality which is new for this model, and is of relevance for our main result (cf. \cref{sec:LR for Villain}).   It also allows to prove the convergence of
the model's correlation function in the  infinite-volume limit through the known $XY$ Ginibre inequality.  Hitherto the latter was not recognized to be applicable to this model.

The model has two phases which differ in the nature of their
spin-spin correlations.  At high temperatures, as is typical in statistical mechanics,  the correlations decay
exponentially fast.   Their low-temperature \emph{BKT phase} is characterized by  the persistence of power-law \emph{slow decay}
 (without  long-range order which is ruled out for the two-dimensional system by the Mermin--Wagner theorem).

Our second result is the following  link between the two phenomena described above.
\begin{thm}  For the Villain model on a doubly-periodic and tame graph  $\G$, at any $\beta$
\be \label{spin_n_relation}
\langle  \sigma_u \cdot   \sigma_v \rangle^{Vill, \G}_{\beta} \geq \Pr{}^{\G^*}_{\lambda = 1/\beta}\left( u \stackrel{n}{\longleftrightarrow} v \right)
\ee
where on the right is  the probability that the two  sites at which the spins are evaluated lie on a common level line   (a concept defined more carefully below) of the   random integer-valued height  function $n$ of the model's dual $\ZGF$.
\end{thm}

This relation  is then used to deduce our main result for the Villain model.
\begin{thm} \label{thm_Villain_LB} Let $\G$ be a planar, doubly-periodic, tame graph. Then for any $\beta$ at which the dual \ZGFs/ model delocalizes at $\lambda=1/\beta$, the correlation function does not decay faster than $\rm{Const.}/d_{\mathcal{G}}(x,y)$, satisfying
\be \label{box_bound}
\forall L<\infty : \qquad \max_{x \in \G} \left( \sum_{\substack{y\in \G \\  d_{\mathcal{G}}(x,y)=L}}  \langle  \sigma_x \cdot   \sigma_y \rangle^{Vill,\G}_{\beta} \right)  \geq  \, 1 \,.
\ee
where $\langle  \sigma_x \cdot   \sigma_y \rangle$ can also be replaced by the smaller correlation function computed for the system's restriction to
$\{ u\in \G: d_{\mathcal{G}}(x,y)\leq L \}$
\end{thm}

Combined with the proven existence of delocalization (i.e. $\lambda_c (\G^*)>0$)  this  yields a new perspective on the BKT phase of the Villain spin model on planar graphs.

The proof of \cref{thm_Villain_LB} is in two steps: i) using the relation \eqref{spin_n_relation}, we show that, if the dual $\ZGF$ random surface depins, then the spin-spin correlations cannot decay exponentially fast and instead are bounded from below by a certain power-law, ii) simplify and improve the lower bound by establishing the following dichotomy for the Villain model, which is known to be valid for a range of other statistical mechanical systems.

\begin{lem}\label{lem_dichotomy}  (The two-point function dichotomy) At any $0<\beta<\infty$, the spin-spin correlation function of the Villain model on a doubly-periodic and tame graph $\G$, decay either exponentially fast, or not faster than $\rm{Const.}/d_\G(x,y)$. More explicitly:  either
\be
\langle  \sigma_0 \cdot   \sigma_x \rangle^{Vill,\G}_{\beta}  \leq  A(\beta)\, \mathrm{e}^{-d_{\G}(x,y) /\xi(\beta)}
\ee
at some finite $A(\beta)$ and $\xi(\beta)$, or else \eqref{box_bound} holds.
 \end{lem}
The proof is based on our extension of a Simon-Lieb type inequality to the Villain model, by which we join it to other known systems for which such a principle applies (references given below).  Instrumental for that is the presentation of the Villain model as the metric graph limit of an $XY$ model. This relation is presented here in \cref{sec:LR for Villain}.

After presenting the above-listed results for the Villain and the $\ZGF$ models, we extend the analysis to other dual pairs including,  in particular,
the  $XY$ spin model and its dual random height function $\ZBF$.

\subsection{Relations with previous works}

The first rigorous proofs of the occurrence in two dimensions of delocalization of random integer-valued height functions,  and the BKT  slow decay of correlations, were accomplished in a groundbreaking work by
 J. Fr\"ohlich and T. Spencer \cite{FroSpe81}.   In it, they dealt with functional integrals in the form of  a two-dimensional Gaussian field modified through the presence of a gas of Coulomb-like point charges with logarithmic interaction.  For such systems, the relevant transition is the condensation at low-temperatures of the 2D Coulomb plasma into a gas of dipoles.
The multi-scale analysis, and the rigorous renormalization group treatments,  that were employed there have inspired a number of other applications, and continue to draw refinements~\cite{GarSep20,DarWu20,BauParRod22A, BauParRod22B}.  Yet, interest continued in alternative approaches to the two phase transitions.

Proofs of delocalization by other methods have been established for a growing collection of random surfaces over two-dimensional graphs.  Among those are:
\begin{itemize}
\item
The dimer model~\cite{MR1872739, MR1933589} and the uniform spanning tree (for certain domains in $\Z^2$~\cite{MR3781449, russkikh2018dominos}, the honeycomb lattice~\cite{MR2415464} and for planar graphs with some natural restrictions~\cite{berestycki2016universality}).
\item
The $F$-model (the six-vertex model with $a=b=1$) at the parameter $c=2$~\cite{MR3592746, glazmanpeled2018}, at $c=1$~\cite{chandgotia2021delocalization, glazmanpeled2018, duminil2019logarithmic} (the height function of square ice), around the free fermion point $c=\sqrt 2$~\cite{dubedat2011exact,MR3637384} and then in the wider range $1\le c\le 2$~\cite{duminil-copin_delocalization_2020}.
\item
The Lipschitz height function model on the triangular lattice (the loop $O(2)$ model) at the parameter $x=1/\sqrt{2}$~\cite{duminil2017macroscopic} and $x=1$~\cite{glazmanManolescu2018}.
\item The above-mentioned work of Lammers~\cite{Lam21}, for a wide range of integer-valued height functions on doubly-periodic planar graphs of maximal degree $3$.
\item For the Solid-On-Solid model on planar graphs~\cite{LamOtt21}.
\end{itemize}

The lattice inequalities of Regev and Stephens-Davidowitz~\cite{RegDav17} were originally motivated by  applications in theoretical computer science.
Here, they are used to show that the height fluctuations of the $\ZGF$ are monotone decreasing  in the coupling constants, and, more generally, in the quadratic form by which the coupling is defined. This notion of monotonicity is conceptually distinct from the FKG inequalities. As far as the authors are aware, this type of monotonicity has not been shown for height functions in the past. On the spin side, analogous statements have been derived through the Ginibre inequalities.

The bound \eqref{spin_n_relation}, which links the models,  is reminiscent of known  relations between the correlation functions of certain  quantum spin systems with the connectivity probabilities of corresponding random loop ensembles \cite{AiNa94,Uel13,AiDCWa20}.   It is then by a common argument that  the a.s. existence of infinite collections of nested loops implies the
  non-summability of the spin correlations $\langle \sigma_0 \sigma_x\rangle_\beta$ .
However, unlike the identities which hold in the quantum examples, the relation \eqref{spin_n_relation}  is only an inequality. Related to that is the unsatisfactory fact that the bounds which are derived here do not yet yield the full temperature dependence of the decay rate.

The dichotomy stated in \cref{lem_dichotomy} is based on the following inequality proven here for the Villain model (cf. \cref{sec:LR for Villain}).
\begin{thm}  For the Villain model on any locally finite graph, let $x,y\in \mathcal V$,  and let $B\subset \G$ be a finite set which separates the two sites (i.e. any path from $x$ to $y$ along the graph's edges intersects $B$).  Then, at any $\beta \in (0, \infty)$,
\be
\langle  \sigma_x \cdot   \sigma_y \rangle^{Vill,\G}_{\beta}     \  \leq  \ \sum_{u\in B}
\langle  \sigma_x \cdot   \sigma_u \rangle^{Vill,\G\cap B}_{\beta}  \, \,
\langle  \sigma_u \cdot   \sigma_y \rangle^{Vill,\G}_{\beta}
  \  \leq  \ \sum_{u\in B}
\langle  \sigma_x \cdot   \sigma_u \rangle^{Vill,\G}_{\beta}  \, \,
\langle  \sigma_u \cdot   \sigma_y \rangle^{Vill,\G}_{\beta}  \,,
\ee
where the superscript $\G\cap B$ indicates that the spin-spin correlation is computed for the finite graph $\G\cap B$ with the free boundary conditions at $\partial B$.
\end{thm}

Inequalities of this type  have been known for a number of spin systems, for which they were derived in a quick succession, starting with the Ising model's Simon's inequality~\cite{Sim79} and its improvement by Lieb~\cite{Lie80}.
For the two-component XY spin model such an inequality was derived in the combination of works by Lieb  and Rivasseau \cite{Riv80}, the latter proving the relevant combinatorial conjecture of \cite{Lie80}.   A slightly different version was presented in Aizenman-Simon \cite{AizSim80B} (proven for $N\leq 4$ component models through Ward identities and inequalities of Simon \cite{Sim80}).
The present extension of the Lieb-Rivasseau inequality  to the Villain model is one of the benefits
of the model's presentation as the restriction to the graph sites of its continuum metric-graph version.

  Given the constructive role which Barry Simon's works  have  played in this area, our results on the direct link of the BKT slow decay of correlations with the depinning in $\ZGF$ were presented  with a dedication to him at the Caltech Simon-fest of April 2021.
  It may be added  that, contemporaneously with the actual posting of our paper, an alternative derivation of the BKT slow decay of correlations in the $XY$ model  was presented in \cite{vanEngLis21}.

 \bigskip
  	
\section{Depinning in the \ZGFl/}
\label{sec:depinning beyond degree 3}

The main result in this section is the proof of \cref{thm:deloc on Z^2}.

The $\ZGF$ model and its two phases were presented in the introduction.
Its partition function is invariant under uniform shifts of the  boundary conditions by  an integer $k$,  to $n_u =k$ at all $u\in \partial \calV_L$,   in which case the corresponding Gibbs state will be shifted accordingly.  Uniform  shifts therefore represent  potential symmetry of the system's infinite-volume Gibbs states.  In the pinned phase this symmetry is broken.

On any graph, the height fluctuations of the $\ZGF$ are upper bounded by those of the corresponding Gaussian free field (GFF)~\cite{FroPar78} (cf.~\cite{kharash2017fr}) in the sense that for every $x,y\in\calV_L$,
\be\label{eq:Gaussian domination}
\EE^{\ZGF,\calG}_{\lambda,L}\left[(n_x - n_y)^2\right]\le \EE^{\text{GFF},\calG}_{\lambda,L}\left[(\phi_x-\phi_y)^2\right]\,.
\ee
Furthermore, for every $f:\calV_L\to\R$,
\be\label{eq:exponential Gaussian domination}
\EE^{\ZGF,\calG}_{\lambda,L}\left[\exp(\left<f,n\right>)\right]\le \EE^{\text{GFF},\calG}_{\lambda,L}\left[\exp(\left<f,\phi\right>)\right].
\ee

The limit in~\eqref{eq_pinning} exists by the model's absolute-value FKG property, which was established by Lammers--Ott~\cite{LamOtt21}. By~\eqref{eq:Gaussian domination}, if~\eqref{eq_pinning} holds for a single $x$ then it holds for all $x$.

Quite generally, the marginal distributions of $(n_x)$ are log-concave.    This important statement, which was proven by
Sheffield~\cite[Section 8.2]{Sheff05}, implies that when \cref{eq_pinning} is violated not only does the second moment diverge, but the entire support of the probability distribution drifts to infinity, in the sense that for any $t<\infty$,
\be  \label{de_pin}
 \lim_{L\to \infty} \,\,  \PP_{\lambda,L}^{\ZGF,\calG}\left[\left\{|n_x| \leq t\right\}\right]  \ = \ 0   \, .
\ee
In particular,
 for every $x\in\calV$,
\be  \label{1st moment}
  \lim_{L\to \infty} \,\,  \EE^{\ZGF,\calG}_{\lambda,L}\left[ |n_x| \right] \ = \ \infty   \,.
\ee

However, it still the case that, in both the pinned and depined phases, the distribution of the gradients of $n$, or equivalently of each of the difference variables  $\{n_x-n_y\}_{x,y\in \calV}$,  remains tight (though not uniformly in $\|x-y\|$).
In this sense, the system behaves similarly to the lattice
$2D$ Gaussian free field, free of the $\ZZ$-valued constraint.

The GFF does not undergo a phase transition since, by its Hamiltonian's homogeneity,  the field's distributions at different temperatures are related via simple rescaling by  $\sqrt{\lambda}$.  In contrast, for the $\ZGF$ there are two natural scales:  that of the discretization step (here $1$), and that of the typical local fluctuations, $1/\sqrt{\lambda} $.  Alternatively stated: there is an energy gap.  This  sets a  relevant scale for $\lambda$,  opening the possibility for a phase transition.

The delocalized phase for the $\ZGF$ on doubly-periodic and tame graphs was first established by Fr\"ohlich and Spencer \cite{FroSpe81} using the dual Coulomb gas perspective, and a multiscale renormalization analysis.  The argument presented here builds on the more recent proof by Lammers~\cite{Lam21} for shift-invariant planar graphs of maximum degree $3$.
The degree limitation arises since in Lammers' argument the persistence of the boundary influence, that is essential for pinning,  for $\lambda $ small is shown to require  simultaneous percolation of positively and of negatively marked edges.  On planar connected  cubic graphs these two random subgraphs block each other, so that both cannot percolate (see \cref{fig:degree 4 percolation coexistence}).

\begin{figure}[t]
	\begin{subfigure}{.4\textwidth}
		\centering
		\begin{tikzpicture}
			
			\draw[very thick,line width=0.1cm,-,blue] plot [smooth] coordinates { (0,0) (0,-1)};
			\draw[very thick,line width=0.1cm,-,red] plot [smooth] coordinates { (-1,0) (0,0)};
			\draw[very thick,line width=0.1cm,-,black] plot [smooth] coordinates { (0,0) (1,1)};
			\draw[very thick,line width=0.1cm,-,red] plot [smooth] coordinates { (1,1) (2,1)};
			\draw[very thick,line width=0.1cm,-,blue] plot [smooth] coordinates { (1,1) (1,2)};
			\node at (0.8,0.3) {$?$};
			\node at (0.3,-0.5) {$-$};
			\node at (-0.5,0.2) {$+$};
			\node at (1.5,1.2) {$+$};
			\node at (1.3,1.8) {$-$};
		\end{tikzpicture}
		\label{fig:degree 3 percolation no coexistence}
	\end{subfigure}
	\begin{subfigure}{.35\textwidth}
		\centering
		\begin{tikzpicture}[scale=0.7]
			\node at (1.6,0.8) {$-$};
			\node at (1.6,3.4) {$-$};
			\node at (0.4,2.2) {$+$};
			\node at (3.4,2.2) {$+$};
			\draw[very thick,line width=0.1cm,-,red] (2,2) -- (0,2);
			\draw[very thick,line width=0.1cm,-,red] (4,2) -- (2,2);
			\draw[very thick,line width=0.1cm,-,blue] (2,0) -- (2,2);
			\draw[very thick,line width=0.1cm,-,blue] (2,2) -- (2,4);
		\end{tikzpicture}	
	\end{subfigure}
	
\caption{The significance of the vertex degree for edge percolation.  Through a vertex of degree $3$ in any configuration only one of the $\pm$ signs may edge percolate.  That is not the case at vertices of higher degree.  An essential part of Lammers' argument is the reduction  of influence percolation at small $\lambda $ to an edge percolation process.}\label{fig:degree 4 percolation coexistence}
\end{figure}
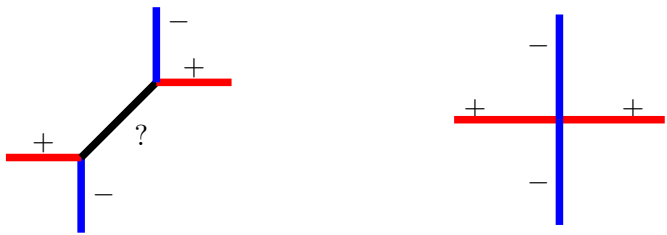

To make Lammers' result applicable to more general graphs,
we transform the given $\ZGF$  model through a sequence of local steps.

To convey the ideas in a notationally-simple context, we concentrate on $\ZZ^2$; the proof presented here extends to the more general cases of doubly-periodic planar graphs, with a tame embedding in $\R^2$.  (The resulting bounds on the critical coupling would, however, depend on the structure of $\G$, in particular its maximal degree).

The first, and simplest, step is based on  the observation that, by the divisibility of its distribution, a Gaussian variable of variance $\frac{1}{\lambda}$ can be presented as the sum of two independent Gaussian variables, of variances $\frac{2}{3\lambda}$ and $\frac{1}{3\lambda}$.  Thus, the $\ZGF$ distribution on any graph remains unaffected if the coupling across any of its edges is replaced by coupling through an intermediate site added in the  midst of the edge.
We  proceed as follows.

\begin{enumerate}[1)]
	\item Enhance the model by adding \emph{real-valued} variables in the midpoints of the edges of $\ZZ^2$, replacing the original interaction by one mediated by mid-edge spins with coupling at strength $3\lambda/2$  across the lower and left parts of the split edges, and $3\lambda$ across the upper and the right parts of the split edges. (This does not affect the distribution of the variables at the original vertices).
	\item Next, constrain the mid-edge variables to take values in $\ZZ$ instead of $\RR$.  It need not be initially  obvious, but is proven below, that this change can  only reduce the fluctuations of $n$.
	\item Collapse the variables pairwise by forcing the pairs of variables north and east from the original lattice sites to assume equal  value (see \cref{fig:local surgery to convert Z^2 to degree-3,fig:resulting degree-3 graph}). This again only reduces fluctuations, and results in a $\ZGF$ model on a cubic planar graph  with uniform coupling constants of $3 \lambda$.
\end{enumerate}

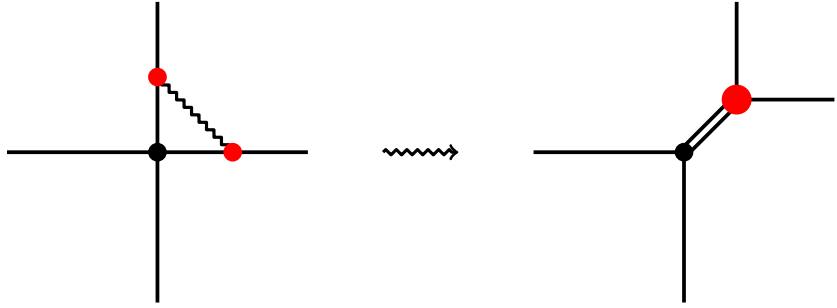
\begin{figure}[t]
	\centering
	\begin{tikzpicture}
		\draw [-,line join=round,very thick, decorate, decoration={
	zigzag,
	segment length=4,
	amplitude=1,post=lineto,
	post length=2pt
}]  (1,0) -- (0,1);

		\draw [->,line join=round,very thick, decorate, decoration={
	zigzag,
	segment length=4,
	amplitude=1,post=lineto,
	post length=2pt
}]  (3,0) -- (4,0);
		
		\node at (0,0)[circle,fill,inner sep=2.5pt]{};
		\draw[very thick,line width=0.05cm,-] (-2,0) -- (2,0);
		\draw[very thick,line width=0.05cm,-] (0,-2) -- (0,2);
		\node at (1,0)[circle,fill,inner sep=2.5pt,red]{};
		\node at (0,1)[circle,fill,inner sep=2.5pt,red]{};

		\node at (7,0)[circle,fill,inner sep=2.5pt]{};
		\draw[very thick,line width=0.05cm,-] (7,-2) -- (7,0);
		\draw[very thick,line width=0.05cm,-] (7,0) -- (5,0);
		\draw[very thick,line width=0.05cm,-] (7,0.08) -- (7.7,0.78);
		\draw[very thick,line width=0.05cm,-] (7,-0.08) -- (7.7,0.62);
		\draw[very thick,line width=0.05cm,-] (7.7,0.7) -- (7.7,2);
		\draw[very thick,line width=0.05cm,-] (7.7,0.7) -- (9,0.7);
		\node at (7.7,0.7)[circle,fill,inner sep=4pt,red]{};
		
	\end{tikzpicture}

	\caption{Local surgery in which a vertex of degree $4$ is replaced by two of degree $3$.  It proceeds through the addition of mediating vertices (at increased coupling strength) in the midst of a pair of adjacent edges, with variables which  are initially $\RR$ and then converted to $\ZZ$-valued.  Subsequently the sites are merged  through  infinite coupling, or equivalently conditioning on having  equal values.  An essential element in our analysis is the proof that in each step the model's fluctuations can only decrease.
	}
	\label{fig:local surgery to convert Z^2 to degree-3}
\end{figure}

Applied to the $\ZGF$ model on $\Z^2$, this construction results in $\ZGF$ on the hexagonal graph, with $\lambda$ increased to $3\lambda$. The scheme for more general graphs is presented in \cref{sec:minorization for general graphs}.

We next establish the monotonicity of fluctuations in these steps, applying for this purpose the monotonicity theory of Regev and Stephens-Davidowitz \cite{RegDav17} which concerns Gaussian probability measures on lattices.
First let us alert the reader that, in the terminology which follows, the relevant lattice $\calL$  is the collection of configurations of  the $\ZGF$  over the given finite graph. Its dimension $k$ is the number of graph vertices.

A lattice $\calL$ in $\RR^k$ is a set of the form $\calL := L\ZZ^k$ for some $L\in\Mat_k(\RR)$. Given such $L$ and a $k\times k$ matrix $A$ which is positive on its span, a random field $\psi\in\calL=L\ZZ^k$ associated with $\left(A,\calL\right)$ is defined via the  partition function
\be
 Z_{A,\calL} := \sum_{\psi\in\calL}\exp\left(-\frac12\ip{\psi}{A\psi}\right)\,.
 \ee
For example, the $\ZGF$ with Dirichlet boundary conditions on a finite graph with the set of vertices $\Lambda$ and coupling $\lambda$ is obtained as the special case of this family of models by choosing $k=|\Lambda|$, $L$ as the orthogonal projection onto vectors $v\in\RR^k$ obeying $\left.v\right|_{\partial\Lambda}=0$ and letting $A = -\lambda \Delta$ where $\Delta$ is the discrete Laplacian on $\ell^2(\Lambda)$ with Dirichlet boundary conditions.

Since, for any $v\in\RR^k$, monotonicity of the moment-generating function implies monotonicity of the second moment of $\ip{v}{\psi}$, let us focus on the former, which we denote:
\begin{align} \MM_{A,\calL}\left[v\right] := \EE_{A,\calL}\left[\exp\left(\ip{v}{\psi}\right)\right]\qquad(v\in\RR^k)\,. \end{align}

In the application we have in mind, $v$ is proportional to $\delta_x-\delta_y$ ($y$ possibly being at the boundary, with $n(y)=0$ fixed by the boundary conditions).

The first of  the two  statements of RSD which are of relevance here is the monotonicity in the lattice:
\begin{prop}[\cite{RegDav17}] \label{sub_latt_monotonicity}
	In the above setup, if $\calM\subseteq\calL$ is a sub-lattice then for any $v\in\RR^k$ $$ \MM_{A,\calM}\left[v\right]\leq \MM_{A,\calL}\left[v\right] \,.$$
\end{prop}
The next states that  $A\mapsto\MM_{A,\calL}\left[v\right]$ is matrix-monotone:
\begin{prop}[\cite{RegDav17}] \label{prop:monotonicity in covariance matrix}
	For any lattice $\calL$ and any two matrices $A,B$ such that $A\geq B>0$ (on the span of $L$), $$ \MM_{A,\calL}\left[v\right] \leq \MM_{B,\calL}\left[v\right]\qquad(v\in\RR^k)\,.$$
\end{prop}

For the reader's convenience, a self-contained presentation of the relevant arguments is presented in \cref{sec:RSD monotonicity theory} below. These two propositions will next be applied in the proof of the results stated above.

\begin{proof}[Proof of \cref{thm:deloc on Z^2}] The first assertion is a direct consequence of the monotonicity in $\lambda$ which follows from \cref{prop:monotonicity in covariance matrix}.

That $\lambda_c(\ZZ^2)<\infty$ is a known fact, provable by a Peierls-type argument. To prove the positivity of $\lambda_c(\ZZ^2)$, we follow the above construction, establishing along the way that in each step the fluctuations, $\EE\left[ \ip{v}{n}^2\right] $ for $v\in\RR^k$,
can only decrease. In the following discussion $\calG = (\calV, \calE)$ refers to the finite-volume restrictions of $\ZZ^2$.

 \emph{Step 1:}  As outlined above, using the identity \begin{align}\label{eq:Interpolation} \exp\left(-\frac12\lambda(n_x-n_y)^2\right) = \frac32\sqrt{\frac{\lambda }{\pi}}\int_{n_{xy}\in\RR}\exp\left\{-\frac\lambda2\left[3(n_x-n_{xy})^2/2+3(n_y-n_{xy})^2\right]\right\}\dif{n_{xy}} \end{align}
we may associate to each edge of $\calG$ an additional \emph{real-valued} degree of freedom $n_{xy}$, so that altogether the field has elastic energy which is associated with the discrete Laplacian on the union graph $\tilde{\calG}$ with vertex set $\tilde{\calV}:=\calV\cup\calE$ and the field is a map $n \in \ZZ^\calV \times \RR^\calE$, with the inhomogeneous coupling constants set at $3 \lambda/2$ and $3\lambda$ as explained above.
 We denote the associated probability distribution $\PP_{\ZZ^\calV \times \RR^\calE,3\lambda;\frac32\lambda}^{\tilde{\calG}} $.  One should note that at this point it is of mixed $\RR$-valued and $\ZZ$-valued  variables.

\emph{Step 2:} The $\RR$-valued field can be viewed as a $\ve\to0$ limit of a $\ve\ZZ$-valued one, that is,
	\begin{align} \lim_{\ve\to0} \MM_{\ZZ^\calV \times (\ve\ZZ)^\calE,\bar{\lambda}}^{\tilde{\calG}}\left[v\right] = \MM_{\ZZ^\calV \times \RR^\calE,\bar{\lambda}}^{\tilde{\calG}}\left[v\right]\qquad(v\in\RR^\calV)\,, \end{align}
where $\bar{\lambda}$ indicates the inhomogeneous coupling constants described above. Applying \cref{sub_latt_monotonicity} to the sequence $\ve_k = 2^{-k}$, in which case the corresponding lattices are  nested ($(2\ve\ZZ)^\calE$ being a sub-lattice of $(\ve\ZZ)^\calE$),  we conclude  that
\be \MM_{\ZGF,\bar{\lambda}}^{\tilde{\calG}}\left[v\right] \leq \MM_{\ZZ^\calV \times \RR^\calE,\bar{\lambda}}^{\tilde{\calG}}\left[v\right]\qquad(v\in\RR^{\tilde{\calV}})\,. \ee
It follows that  the conversion of the mixed $\RR$-valued and $\ZZ$-valued Gaussian field to a pure $\ZGF$ on $\tilde{\calG}$ with the same coupling constants  can only lower the variance of $n_0$.

\begin{rem}\label{rem:Gaussian domination bound}We note in passing that the above gives another proof of the Gaussian domination assertions stated in~\eqref{eq:Gaussian domination} and~\eqref{eq:exponential Gaussian domination}. Explicitly,
\begin{align}\label{eq:Gaussian domination bound}\MM_{\lambda}^{\ZGF,\calG}\left[v\right]\equiv\EE_{\lambda}^{\ZGF,\calG}\left[\mathrm{e}^{\langle v,n\rangle}\right]\leq \exp\left(+\frac{1}{2\lambda}\langle v,-\Delta^{-1} v\rangle\right)\qquad(v\in\ker(-\Delta)^\perp)\,.\end{align}
\end{rem}

\emph{Step 3:}
In a repeated application of the monotonicity principle expressed in \cref{sub_latt_monotonicity}  we deduce that also in the third step of the construction,  the restriction to the sub-lattice of configurations in which the new edge variables are pairwise equal,  the fluctuations can only decrease.
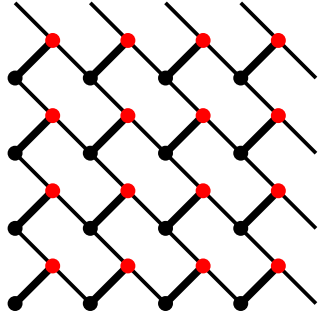
\begin{figure}[t]
	\centering
	\begin{tikzpicture}[scale=1]
		\foreach \x in {0,...,3}
		{
			\foreach \y in {0,...,3}
			{
				\draw[very thick,line width=0.1cm,-] (\x,\y) -- (\x+0.5,\y+0.5);
				\draw[very thick,line width=0.05cm,-] (\x+1,\y) -- (\x,\y+1);
				\node at (\x ,\y )[circle,fill,inner sep=2pt]{};
				\node at (\x + 0.5,\y + 0.5)[circle,fill,inner sep=2pt,red]{};
							}
		}	
	\end{tikzpicture}	
	\caption{The degree-$3$ graph on which the fluctuations of $n_0$, at adjusted couplings, provide a lower bound for those on $\Z^2$.  The  (thick) positively-sloped lines result from the binding of two edges, at the construction's last step.   The splitting ratio of the couplings at the first step is selected so as to equilibrate the resulting model's coupling.  These end up at $3 \lambda$, in terms of the $\ZGF$'s coupling  on the original graph  $\Z^2$. }
	\label{fig:resulting degree-3 graph}
\end{figure}

The assertion that the $\ZGF$ on $\Z^2$ is depinned at   $\lambda\le\frac{2}{3}\ln(2)$  follows by applying  Lammers' theorem to the hexagonal graph which is obtained by the above construction.

An extension of this procedure to more general doubly-periodic graphs is given below in \cref{sec:minorization for general graphs}.
\end{proof}

\section{The Villain and XY spin models}

\subsection{A pair of $O(2)$ symmetric spin models}
	Prototypical examples of   ferromagnetic systems with $O(2)$ symmetry are arrays of two-component unit spins, attached to the sites $\calV$ of a homogeneous graph $\calG$ with, e.g., the XY ferromagnetic  interaction
	\be \label{XY_ferro}
	H^{\mathrm{XY}} ( \sigma) = - \sum_{\{u,v\}\in\calE} J_{u,v} \, \sigma_u \cdot   \sigma_v  =
	- \sum_{\{u,v\}\in\calE} J_{u,v} \, \cos(\theta_u-\theta_v)\,.
	\ee
at $J_{u,v} \geq 0$.
The spins are presented here in two equivalent ways: as two-component unit vectors $ \sigma_x \in \mathbb{S}^1$,  or as complex unitary numbers  $\mathrm{e}^{i \theta_x}\in\mathbb{S}^1\subseteq\CC$.

Since the phenomena discussed here are unique to two dimensions, the graphs $\calG$ of main interest  are doubly-periodic and tame. The graphs may be finite or infinite, the latter case conveniently viewed as the infinite-volume limit of its finite restrictions.   Thus $\Z^2$ is studied as a limit ($L\to \infty$) of its truncated version $\calV_L = [-L,L]^2\cap\ZZ^2$.
By default we take the model with free boundary conditions.

The basic $O(2)$ spin model, commonly referred to  as the $XY$ model,  has as its  partition function (at the nearest-neighbor interaction and free boundary conditions)
	\begin{align}  \label{Z_XY}
	Z^{XY,\G}_{\beta} & =
	  \int_{\left[-\pi,\pi\right)^\calV}
	\prod_{\{u,v\}\in\calE} \re^{\beta    J_{u,v} \cos(\theta_u-\theta_v)} 	
\prod_{x\in \mathcal V}  \dif{ \theta_x}
	\, \notag \\  \\ \notag
	 &= \re^{\beta \|J\|} \int_{(\mathcal S^1)^\calV }
	\prod_{\{u,v\}\in\calE} \re^{-\frac{\beta   J_{u,v}}2 \|\sigma_u-\sigma_v\|^2 }
\prod_{u\in \mathcal V} \dif{ \sigma_u}
 \, .
	\end{align}
with $\|J\| =\sum_{(x,y)\in \mathcal E} J_{x,y}$.

In the Villain version of the $O(2)$ model~\cite{Vil75} each of the interaction factors in  \cref{Z_XY}  is replaced through the following substitution
	\be  \label{Villain_subst}
	\re^{\beta  J_{u,v} \cos(\theta_u-\theta_v)}   \quad \Longrightarrow \quad     \sum_{m_{uv}\in\ZZ } \re^{-\frac{\beta   J_{u,v} }{2}   \left(\theta_u-\theta_v +2\pi m_{uv}\right)^2}\,.
	\ee
The resulting partition function is presented in \eqref{Villain_Z1}.

In each case, the corresponding Gibbs equilibrium states is given by the probability distributions over
 the spin configurations that is obtained by normalizing the measures integrated in \cref{Z_XY}  and \cref{Villain_Z1}   through the  corresponding partition functions.

In what follows, we shall focus our attention on the ferromagnetic spin models with the  nearest-neighbor couplings  $J_{u,v} = \1[d_\G(u,v)=1] $.

The $XY$ and the Villain spin systems are similar in terms of the physics they model:\vspace{-0.06in}
\begin{enumerate}[i)]
\item in each the  systems' symmetries include uniform  spin rotations
\item the weights assigned to the spin configurations have the same periodicity in $\theta$ and same maxima, in the vicinity of which they agree up to the fourth order (after a trivial correction by  factors of  $\re^\beta$).
\end{enumerate}
Of the two, the Villain version has a simpler dual, the duality being accomplished through a Fourier transform aided by the Poisson summation formula.

In addition, as is  shown and applied below,  the Villain model can be presented as a metric-graph continuum limit of an $XY$ spin system.  There is also a mathematical relation in the converse direction (see \cref{lem:ZBF is a mixture of Gaussians}).

Of particular interest is the behavior of the correlation function
\be
\langle  \sigma_x \cdot   \sigma_y \rangle_{\beta}  = \lim_{L\to \infty}
\langle  \sigma_x \cdot   \sigma_y \rangle_{\beta,L}
\ee
where  $\sigma_x \cdot   \sigma_y  =  \cos(\theta_x-\theta_y)$
and $\langle \cdot \rangle_{\beta,L}$ denotes the expectation value with respect to the  finite-volume Gibbs equilibrium state in $\Lambda(L)=[-L,L]^2$ at the inverse temperature $\beta = (k_B T)^{-1}$ with free boundary conditions.  The limit's existence for both the basic $XY$ model and its Villain version follows by the Ginibre correlation inequalities \cite{Gin70, FroPar78}.  These imply that $\langle  \sigma_x \cdot   \sigma_y \rangle_{\beta,L}$ is  monotone increasing in $L$, and more generally in the volume.
	
	By  general arguments~\cite{Dob68,Sim79},
	 or more model-specific bounds~\cite{AizSim80A}, at high enough temperatures  the spin-spin correlations decay exponentially fast.  That is,  for all $\beta<\beta_c$, at some $\beta_c >0$,
		\be  \label{eq:highT}
\langle  \sigma_x \cdot   \sigma_y \rangle_{\beta}    \leq   A(\beta) \Oe^{-m(\beta)\|x-y\| }
	 \qquad \mbox{with $m(\beta)>0$, $A(\beta) < \infty$}\,.
		\ee
	
	   In dimensions  $d>2$, such models  exhibit also a low-temperature phase with continuous symmetry breaking, where $\lim_{|y |\to \infty } \langle  \sigma_x \cdot   \sigma_y \rangle_{\beta}
 >0$~\cite{FrSiSp76}.  However,  by
the Mermin--Wagner theorem, that does not occur in two dimensions~\cite{DobShlo75, Pfi81,McBSpe77}.

Nevertheless, as was pointed out by   Berezinskii~\cite{Ber70} and Kosterlitz and Thouless~\cite{KosThou73}, and proven rigorously in~\cite{FroSpe81}, two-component  spin systems with a rotation-invariant interaction do exhibit a low  temperature phase at which \cref{eq:highT} does not hold.  At that phase, the spin-spin correlations decay slowly, at a temperature-dependent power law, i.e., there is some $\beta_\mathrm{BKT}>0$ such that
	\be
		\langle  \sigma_x \cdot   \sigma_y \rangle_\beta  \approx \frac{\mathrm{Const.}}{\|x-y\|^{\eta(\beta)}} \qquad \mbox{(for all  $ \beta > \beta_\mathrm{BKT}$)}\,.
	\ee

		It is expected that such behavior   does not occur for systems in which the $O(2)$ spin-rotation symmetry is part of a larger  non-Abelian $O(N)$ symmetry at $N>2$, e.g., in the three-component classical  Heisenberg model~\cite{Poly75}.   However, in contrast to the case $N=2$,  still unresolved challenges were raised to the  arguments on which this prediction rests
~\cite{PatSei92}.
		
\subsection{The spin correlation dichotomy}  \label{sec:dichotomy}

The  dichotomy stated in \cref{lem_dichotomy}  is a useful principle which allows to boost initial result of slow decay to a rathe specific lower bound.
As such, it can  be viewed as a simplified version of a general result of Dobrushin and Pecherski~\cite{DobPech79} which for a range of $N$-component ferromagnetic spin models ($N\leq 4$) was simplified through correlation inequalities.
These are typically styled after the Ising prototypes of the Simon inequality~\cite{Sim80}  and  its Lieb-improved version~\cite{Lie80}.

Correlation inequalities from which the dichotomy follows are valid for both the $XY$ and the Villain models.
For the $XY$ model they have been known in two forms, that of Aizenman and Simon~\cite{AizSim80B} (proven for $N\leq 4$), and Lieb and Rivasseau~\cite{Lie80,Riv80} (proven for $N\leq 2$).
 For the Villain model this statement is new, and is proven here in \cref{sec:LR for Villain}.
 Combining the old with the new one has:

\begin{lem}  \label{lem_corr_ineq} For an arbitrary graph $\G$, any finite subset $\Lambda \subset \mathcal V$ and any pair of  sites, $x$ in $\Lambda $ and
$y$ in  its complement $\Lambda^c$, the $XY$ model's correlation functions satisfy
\be \label{XY_AS}
\langle  \sigma_x \cdot   \sigma_y \rangle ^{\G}_{\beta}  \leq
\sum_{u\in \Lambda, v\in \Lambda^c}
\langle \sigma_x \cdot   \sigma_u\rangle ^{\G}_{\beta} \, \beta  J_{u,v}  \,
\langle \sigma_v \cdot   \sigma_y\rangle ^{\G}_{\beta}
\ee
and also
\be \label{XY_LR}
\langle  \sigma_x \cdot   \sigma_y \rangle ^{\G}_{\beta}  \leq
\sum_{u\in \partial_s\Lambda}
\langle \sigma_x \cdot   \sigma_u\rangle ^{\G\cap \Lambda}_{\beta} \,\,
\langle\sigma_u \cdot   \sigma_y\rangle ^{\G}_{\beta}  \, \leq\,
\sum_{u\in \partial_s\Lambda}
\langle\sigma_x \cdot   \sigma_u\rangle ^{\G}_{\beta} \,\,
\langle\sigma_u \cdot   \sigma_y\rangle ^{\G}_{\beta}
\ee
where $\partial_s\Lambda$ is the site (inner) boundary of $\Lambda$.

Furthermore, \eqref{XY_LR} holds also for the Villain model [this being the only new assertion here].
\end{lem}

\cref{lem_dichotomy} follows through a simple iteration, of either  \eqref{XY_AS} or  \eqref{XY_LR}.  (The two also open paths for the principle's extensions to two slightly different classes of related systems.)

\section{The known duality relations between  \ZGFs/ and the Villain spin model}

\subsection{Equality of  the partition functions}

Restating \eqref{Villain_Z1},
the Villain model's partition function, on a finite graph $\G$ and with the free boundary conditions,   is
\be \label{Villain_Z1_repeated}
		Z^{\Vill,\G}_{\beta} =   \int_{\theta\in [-\pi,\pi)^{\calV}  }
		\prod_{\{u,v\}\in\calE} \left[ \sum_{m_{uv}\in\ZZ } \re^{-\frac{\beta}{2}   \left(\theta_u-\theta_v +2\pi m_{uv}\right)^2} \right]
		   \,
		\dif{ \theta }\, ,
\ee
and the corresponding Gibbs equilibrium expectation values of  local functions of $\{\theta_x\}_{x\in\calV}$
are given by
\be
\langle f \rangle^{\Vill,\G}_\beta =
		\frac{1}{Z^{\Vill,\G}_{\beta}}\int_{\theta\in[-\pi,\pi)^{\calV} }  f(\theta)   \, \,  \prod_{\{u,v\}\in\calE} \left[ \sum_{m_{uv}\in\ZZ } \re^{-\frac{\beta}{2}   \left(\theta_u-\theta_v +2\pi m_{uv}\right)^2} \right]
	\dif{ \theta }\,.
		\ee

In a known duality relation~(c.f. \cite[Appendix A]{FroSpe81})  for planar $\calG$
the Villain model's partition function  equals that of the $\ZGF$ on the dual graph $\calG^\ast$ at $\lambda = 1/\beta $,  with the Dirichlet boundary conditions ($n_u =0 \,\, \forall u\in \partial \calV^*$):
	\begin{align}(2\pi)^{|\calV^\ast|} Z^{\Vill,\G}_{\beta} = \sum_{\substack{n\in\ZZ^{\calV^\ast}\\  n_{\partial \calG^*} \equiv 0}}
	\exp\left(-\frac{1}{2\beta}\norm{\nabla n}^2\right )
\equiv Z_{\lambda=\frac{1}{\beta}}^{\ZGF,\calG^\ast}
	\label{eq:Villain_ZGF_partition}
	\end{align}
where $  \nabla n $ (the gradient  of $n$)  is a function defined over the edges of $\calG^\ast$, associating to the oriented edge $(x,y)$ the difference $\nabla n(x,y) = n(y)-n(x)$, and
	$$ \norm{\nabla n }^2 \equiv \sum_{\{x,y\}\in\calE^\ast}(n_x-n_y)^2\,. $$

To clarify the notation when discussing the dual models on the dual pair of graphs we adapt the following convention: The vertices of $\G$ will continue to be called vertices, while the vertices of $\G^\ast$, on which the $\ZGF$ model is defined, will be referred to as faces.

\subsection{The dual expression for the spin  correlation function}

Under the above correspondence the spin-spin correlation function is mapped onto what may be viewed as the expectation value of a dislocation-favoring operator.   In presenting it  we shall employ the following notation: given a pair of sites $\{x,y\} \subset \calV$ and an oriented path $\gamma_{yx}$ from $x$ to $y$ along the edges of $\calE$, we denote by  $\Gamma_{yx}$   the  function defined over the oriented edges of $\calG^\ast$ by
	\[
	\Gamma_{yx}(u,v) := \begin{cases} +1\, (\text{resp.} -1) & (u,v) \text{ is a counter-clockwise (resp. clockwise)}\\&\quad\quad\qquad\text{rotation of an edge  traversed by $\gamma_{yx}$} \\
		0 & \mathrm{otherwise} \end{cases}\,.
	\]
The following expression represents a dislocation-favoring deformation of the elastic energy, which favors a jump discontinuity across $\gamma_{yx}$:
\be  \norm{\nabla n+\Gamma_{yx}}^2 \equiv \sum_{\{u,v\}\in\calE^\ast}\big(n_u-n_v+\Gamma_{yx}(u,v)\big)^2 \,.
\ee
		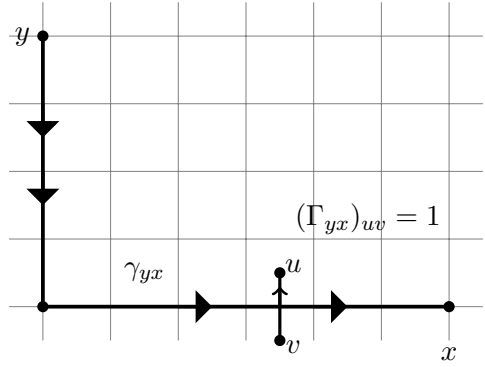
\begin{figure}[t!]
		\centering
		\begin{tikzpicture}[scale=0.9]
			\draw[step=1cm,gray,very thin] (-0.5,-0.5) grid (6.5,4.5);
			\draw[very thick,line width=0.05cm,-] (0,0) -- (6,0);
			\node at (0,4)[circle,fill,inner sep=1.5pt]{};
			\node at (0,0)[circle,fill,inner sep=1.5pt]{};
			\node at (-0.3,4) {$y$};
			\node at (6,-0.7) {$x$};
			
			\node at (1.5,0.5) {$\gamma_{yx}$};
			\draw[very thick,line width=0.05cm,-] (0,0) -- (0,4);
			
			\node at (6,0)[circle,fill,inner sep=1.5pt]{};
			
			\node at (3.5,0.5)[circle,fill,inner sep=1.5pt]{};
			\node at (3.5,-0.5)[circle,fill,inner sep=1.5pt]{};
			\node at (3.7,0.6) {$u$};
			\node at (3.7,-0.6) {$v$};
			\draw [-,very thick] (3.5,0.5) -- (3.5,-0.5);
			\draw [->,very thick] (3.5,0) -- (3.5,0.3);
			\node at (4.8,1.3) {$(\Gamma_{yx})_{uv}=1$};
			
			\draw [->,very thick,-triangle 90] (2,0) -- (2.5,0);
			\draw [->,very thick,-triangle 90] (4,0) -- (4.5,0);
			\draw [->,very thick,-triangle 90] (0,2) -- (0,1.5);
			\draw [->,very thick,-triangle 90] (0,3) -- (0,2.5);
		\end{tikzpicture}		
		
		\caption{A depiction of $\gamma_{yx}$ and $\Gamma_{yx}$.}
		\label{fig:gamma_xy}
	\end{figure}

The duality transformation leading to \cref{eq:Villain_ZGF_partition} yields also the following relation~\cite{FroSpe81}, which holds for any choice of the $\pm$ sign:
\begin{eqnarray}
\langle \sigma_x\cdot\sigma_y  \rangle^{\Vill,\G}_{\beta} & \equiv& \langle\mathrm{e}^{\pm i(\theta_y-\theta_x)}  \rangle^{\Vill,\G}_{\beta}   = \quad \frac{\sum_{\substack{n\in\ZZ^{\calV^\ast}  n_{\partial \calG^*} \equiv 0}}\exp(-\frac{1}{2\beta}\norm{\nabla n  \pm \Gamma_{yx}}^2)}{Z^{\ZGF,{\G^{\ast}}}_{1/\beta}}
 \notag  \\  \label{eq:duality_correlations}
&=&    \EE^{\ZGF,{\G^{\ast}}}_{1/\beta} \left[ T_{\gamma_{yx} }^+\right] =  \EE^{\ZGF,{\G^{\ast}}}_{1/\beta} \left[ T_{\gamma_{yx} }^-\right]
\,,
\end{eqnarray}
where $T_{\gamma_{yx}}^+ $ and $T_{\gamma_{yx}}^- $ are the two functions
\be \label{eq:T}
T_{\gamma_{yx} }^\pm(n) = \exp\left(\frac{1}{2\beta}\left(\norm{\nabla n}^2- \norm{\nabla n \pm \Gamma_{yx}}^2  \right)\right)\,.
\ee
	
	One may note the homotopy invariance of the expression on the right in \cref{eq:duality_correlations}:  as a function of $\gamma$, its value is equal for any pair of paths with coinciding end points.  That is obviously true for the LHS of  \cref{eq:duality_correlations}.  For the RHS this can be explained through a ``gauge transformation'' which consists of an increase of each $n_u$ by the winding number around $u$ by the oriented loop formed by concatenating one of the paths with the inverse of the other.
	
\section{A (new) relation of \ZGFs/ level loops with  spin-spin correlations} \label{sec:level_lines}

The duality relation \cref{eq:duality_correlations} will be used below for a key lower bound on the spin correlation function  in terms of the  probability that the two sites lie on a common level line of the $\ZGF$.  We start by defining the relevant concepts.

\subsection{Level lines of the $\ZGF$}

Viewing the $\ZGF$ process as describing  a random height function on the graph $\calG^*$, it is natural to think of its realization in terms of the corresponding contour maps.  Adapting to the fact that the values of $\{n_u\}_u$ may change discontinuously, we  base the contour description on oriented contours of  heights in $\Z +1/2$.

More specifically: with each realization $n$ of the $\ZGF$ on a planar graph $\calG^\ast$, embedded in $\R^2$, we associate the  family of  non-crossing contours, each either a loop or an infinite line, drawn in $\R^2$ through the following rules:
\begin{enumerate}[i)]
\item  an oriented contour line at level $q\in \ZZ+\frac{1}{2}$  passes through an edge $\{u,v\}\in \calE$ if and only if $n<q$ on the face to its   right and $n>q$ at the face to its left
(thus the number of contour lines crossing the edge $\{u,v\}$ is exactly $|n_u-n_v|$, one for each  value of $q$ meeting these conditions).
\item the right-hand rule:   at each vertex of $\calG$ each incoming $q$-line exits through the first outgoing $q$-line to its right in the counterclockwise order of the edges exiting the given vertex, as depicted in \cref{fig:loop resolution using RH rule}.
\end{enumerate}
The construction is enabled by the fact  that at each vertex
the orientations of the incoming and outgoing $q$ lines alternate in the circular ordering of the edges linked to  the vertex.

Segment  of an oriented $q$ contour would be referred to as  $q$-paths.   For close contours, we adapt the standard terminology of calling counterclockwise (resp. clockwise) loops as positively (resp. negatively) oriented.

 For any face $f \in \calV^*$ and $\eta = \pm$,  the number of loops of orientation $\eta$ that surround $f$ will be denoted  $\mathcal{N}^{\eta}_q(f)$, and with $r=\pm$ we further denote
 by  $\mathcal{N}^{\eta,r}(f)$  the number of $\eta$-oriented loops with $\sgn{q} =r$,  i.e.
 \be \label{N+-}
  \mathcal{N}^{\eta, r}(f)  = \sum_{\substack{q\in \frac{1}{2} +\Z \\ \rm{sgn} \, q =r }}  \mathcal{N}_q^{\eta}(f)
 \ee
 From the definition of the level lines it follows that in any configuration
\be \label{eq:LoopsToHeight}
\mathcal{N}^{+,+}(f) + \mathcal{N}^{-,-}(f) \, \geq \,  |n_f| \, .
\ee

\begin{figure}[t]

		\centering
		\begin{tikzpicture}[scale=0.8]
			\draw[step=2cm,gray,very thin] (0,0) grid (4,4);
			\node at (1,1) {$4$};
			\node at (3,1) {$-5$};
			\node at (1,3) {$-6$};
			\node at (3,3) {$7$};
			\draw[very thick,line width=0.1cm,->,blue] plot [smooth] coordinates { (2,0.2) (2,1.5) (2.5,2) (3.8,2)};
			\draw[very thick,line width=0.1cm,->,blue] plot [smooth] coordinates { (2,3.8) (2,2.5) (1.5,2) (0.2,2)};
		\end{tikzpicture}
	\caption{Example of level loops resolution through the right-hand rule.}\label{fig:loop resolution using RH rule}
\end{figure}

	\subsection{Lower bounds on the spin correlation function}

Our first set of new results relate the Villain model's BKT phase to depinning in the $\ZGF$. A key intermediary step involves the event described below.

\begin{defn}\label{def:A_xy^q}  
Let $\gamma_{yx}$ be an arbitrary oriented, simple path in $\calG$ that begins at $y$ and ends at $x$, and $e =(x, x')$ an oriented edge  such that $x'$ is not in $\gamma_{yx}$. For any $q \in \Z + 1/2$, we define as  $A_{\gamma_{yx},e}^q$  the event that there is $q$-path $\kappa^q_{xy}$ which begins with the oriented edge $e$ and terminates at $y$ such that $\gamma_{yx} \circ \kappa^q_{xy}$ is a simple loop with orientation $\mathrm{sgn}(q)$.
\end{defn}

A particular example of the occurrence of $A_{\gamma_{yx},e}^q$ where $\gamma_{yx}$ is the `L'-shaped path and $q$ is positive  is depicted in \cref{fig:schematic of gamma path}. We also need to demarcate the two faces that border the oriented edge $e$ that begins the path $\kappa^q_{xy}$, as they will play a special role in our analysis. Given an oriented edge $e$, we set $x^+$ and $x^-$ to be the faces to the left and right of the edge, respectively.

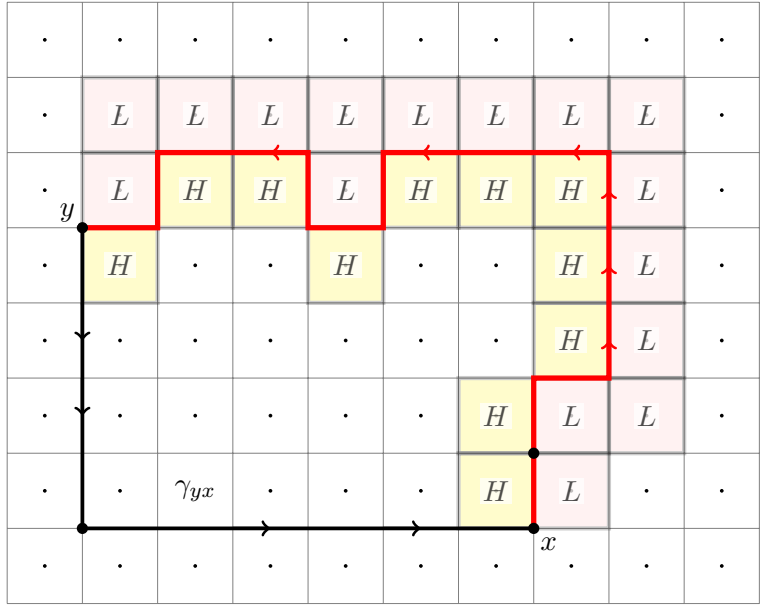
\begin{figure}[t]
	\centering
	\begin{tikzpicture}[scale=1]
		\draw[step=1cm,gray,very thin] (-1,-1) grid (9,7);
		\foreach \x in {-1,...,8}
		{
			\foreach \y in {-1,...,6}
			{
				\node at (\x + 0.5,\y + 0.5)[circle,fill,inner sep=0.5pt]{};
			}
		}
		\foreach \pt in {(5,0),(5,1),(6,2),(6,3),(6,4),(5,4),(4,4),(3,3),(2,4),(1,4),(0,3)}
		{	
			\parsept{\x}{\y}{\pt}
			\draw [ultra thick, draw=black, fill=yellow, opacity=0.2] (\x,\y) rectangle (\x+1,\y+1);
			\node[fill=white,inner sep=1pt,opacity=0.7] at (\x+0.5,\y+0.5) {$H$};
		}
		\foreach \pt in {(6,0),(6,1),(7,1),(7,2),(7,3),(7,4),(7,5),(6,5),(5,5),(4,5),(3,5),(3,4),(2,5),(1,5),(0,5),(0,4)}
		{	
			\parsept{\x}{\y}{\pt}
			\draw [ultra thick, draw=black, fill=pink, opacity=0.2] (\x,\y) rectangle (\x+1,\y+1);
			\node[fill=white,inner sep=1pt,opacity=0.7] at (\x+0.5,\y+0.5) {$L$};
		}
		\draw[very thick,line width=0.05cm,-] (0,0) -- (6,0);
		\draw[very thick,line width=0.05cm,-] (0,0) -- (0,4);
		\draw [->,very thick] (2,0) -- (2.5,0);
		\draw [->,very thick] (4,0) -- (4.5,0);
		\draw [->,very thick] (0,2) -- (0,1.5);
		\draw [->,very thick] (0,3) -- (0,2.5);		
		\draw [very thick, line width=0.07cm,red,-] plot  coordinates {(6,0) (6,1) (6,2) (7,2)(7,5)(4,5)(4,4)(3,4)(3,5)(1,5)(1,4)(0,4)};
		\draw [->,very thick,red] (7,2) -- (7,2.5);
		\draw [->,very thick,red] (7,3) -- (7,3.5);
		\draw [->,very thick,red] (7,4) -- (7,4.5);
		\draw [->,very thick,red] (7,5) -- (6.5,5);
		\draw [->,very thick,red] (5,5) -- (4.5,5);
		\draw [->,very thick,red] (3,5) -- (2.5,5);
		\node at (0,4)[circle,fill,inner sep=1.5pt]{};
		\node at (0,0)[circle,fill,inner sep=1.5pt]{};
		\node at (6,0)[circle,fill,inner sep=1.5pt]{};
		\node at (6,1)[circle,fill,inner sep=1.5pt]{};
		\node[fill=white,inner sep=1pt] at (6.2,-0.2) {$x$};
		\node[fill=white,inner sep=1pt] at (1.5,0.5) {$\gamma_{yx}$};
		\node[fill=white,inner sep=1pt] at (-0.2,4.2) {$y$};
		
	\end{tikzpicture}

	\caption{A schematic depiction of the event $A_{xy}^q$ for a positive value of $q$. The dots indicate the points on which the field $n$ is defined, and
the marks $H, L$ the sites whose $n$ values are uncovered in the $\kappa_{xy}^q$-exploration process.}
	\label{fig:schematic of gamma path}
\end{figure}

We now show that the two point function of the Villain model can be bounded by the probability of the event described above.    It should be emphasized here that, while the path $\gamma_{yx}$ is arbitrary, for the following bound it must be preselected (i.e. not adjusted dynamically to the realization of the field $n$).

\begin{thm}\label{thm:1} Given a dual pair of the Villain model on a finite planar graph $\calG$ and the corresponding $\ZGF$ on $\calG^*$, for any  path $\gamma_{yx}$ and an oriented edge $e= (x,x')$ which does not intersect it,  at any $q\in \Z+\frac{1}{2}$,
\be \label{key_bound}
\langle \sigma_x\cdot\sigma_y \rangle^{\Vill,\calG}_\beta \,  \geq \,
\PP_{1/\beta}^{\ZGF,\calG^\ast}\left[ A_{\gamma_{yx},e}^q \right]\,.
\ee
Furthermore, \cref{key_bound} holds also for $q$ replaced by ${\widetilde q}(n)$ which is  determined by  the values of $n$ on the pair of faces sharing the edge $e$ (i.e. by $\{n(x^+), n(x^-)\}$).
\end{thm} 	

The bound stated in \eqref{eq:main theorem relation} is obtained through a summation of the above, in a method which applies quite generally to  doubly-periodic planar graphs.  To convey the argument in a relatively simple way we present it in the context of $\Z^2$.

By \eqref{key_bound}  fast decay of spin correlations implies that in the corresponding $\ZGF$  long level lines occur only rarely.   For a  more quantitative statement, we next combine this with the observation that any loop which encircles a face $f\in (\Z^2)^*$, at distance $R$ from it, implies that the conditions for the event $A^{{\widetilde q}}_{xy,e}$ in \eqref{key_bound} are met for at least four distinct pairs $\{x,y\}$, with $\|x-y\| \geq R$, for which the difference vectors $y-x$ lie each in a different quadrant of $\Z^2$.

For this purpose, consider the standard quadrants of $\Z^2$ shifted so that their intersection is the face  $f_0\in (\Z^2)^*$ centered at $(1/2,1/2)$ and their corners form the four vertices of $f_0$.    Let  $\widetilde Q_j$ be the collection of pairs of sites of $\{x,y\}$  one  on the vertical and
the other on the horizontal part of the boundary of $Q_j$.  As depicted in \cref{fig:PointsOnRays}, each $q$ loop encircling $f$, of orientation $\sgn{q}$  implies the existence of at least one pair of sites in each $Q_j$ ($j=1,...,4$) which are linked by a path meeting the conditions for $\kappa^q_{xy}$, whose first step is into the corresponding quadrant.
This observation leads to the following statement.

\begin{figure}[t]               \centering
     \begin{subfigure}[b]{0.35\textwidth}
         \centering
         \includegraphics[width=\textwidth,page=6]{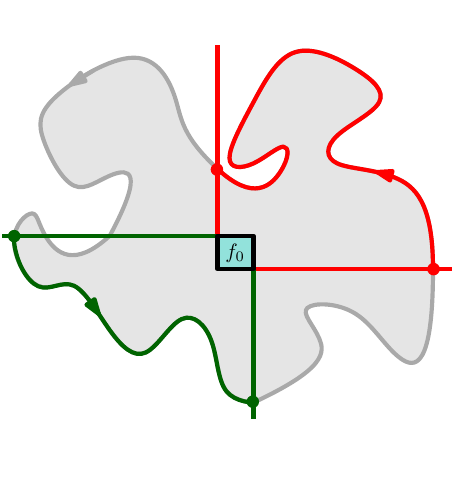}
         \label{fig:RaysFromFace}
     \end{subfigure}
     \quad
     \begin{subfigure}[b]{0.35\textwidth}
         \centering
         \includegraphics[width=\textwidth,page=7]{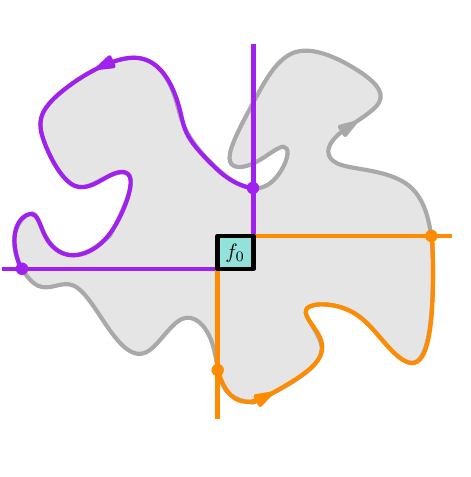}
     \end{subfigure}
              \caption{The quadrants $Q_1$, $Q_2$, $Q_3$ and $Q_4$ (in red, purple, green and orange, respectively), with the realizations of $\kappa^{q}_{xy}$, for $q>0$.}
              \label{fig:PointsOnRays}
\end{figure}

\begin{thm}\label{thm:2} In the notation explained above, for  $\calG_L = \Z^2 \cap [-L,L]^2$,  at  any  $L < \infty$  and  $\ve>0$
	\begin{equation}\label{eq:main theorem relation}
		 \sum_{j=1}^4 \sum_{(x,y)\in \widetilde Q_j} \left(\langle\sigma_x\cdot\sigma_y \rangle^{\Vill,\calG_L}_\beta\right)^{1-\ve} \, \geq \, \frac{4 \cdot C_\ve}{\sqrt{\beta}}\,  \,\EE^{\ZGF,\calG_L^*}_{\lambda}\left[\mathcal{N}^{+,+}(f_0) + \mathcal{N}^{-,-}(f_0)\right]
	\end{equation}
with  $C_\ve>0$ which depends only on $\ve$.
\end{thm}

It may be of interest to note here the similarity of \cref{thm:2} with the relation presented in~\cite[Theorem 6.1]{AiNa94}
between the existence of infinite collections of nested loops in another two-dimensional loop model and the non-summability of the correlation function of a quantum spin system.

The proofs are presented in the next section.    However to  not break the flow
let us note here the link of  the above lower bound with the delocalization proven in \cref{thm:deloc on Z^2}.

By the above mentioned monotonicity of the correlations as function of the volume    \cref{thm:2} implies that in the infinite-volume limit
\be
\sum_{u} \left(\langle\sigma_0\cdot\sigma_u \rangle^{\Vill}_\beta\right)^{1-\ve}  =
\sum_{j=1}^4 \sum_{(x,y)\in \widetilde Q_j}
\left(\langle\sigma_x\cdot\sigma_y \rangle^{\Vill}_\beta\right)^{1-\ve} \geq \frac{C_\ve}{\sqrt{\beta}}\,  \,\EE^{\ZGF,\calG^*}_{1/\beta}\left[|n_{(1/2,1/2)}| \right]\,
\label{eq:Lower bound on sum of two point function}\ee
where use was made of translation invariance (which also follows from the correlations' monotonicity properties), and the substitution $u=y-x$.  As was noted in \cref{1st moment}, the righthand side diverges in the de-pinned phase, proving the claim of slow decay.

\section{Proof of the $\ZGF$ bound on the spin correlation function}

For the rest of this section, we will fix an oriented, simple path $\gamma_{yx}$ and an oriented edge $e = (x,x')$ such that $x'$ is disjoint from $\gamma_{yx}$.

\subsection{Reduction to a conditional expectation}

In this section we prove \cref{thm:1}  taking for granted the positivity of what we term the $\ZGF$'s stiffness modulus.   The proof of this general property of the $\ZGF$ on finite graphs  is presented to the next section.

In preparation for the proof, let us note that from  the duality relation \cref{eq:duality_correlations} one may conclude  the following
\begin{eqnarray}
	\langle \sigma_x \cdot \sigma_y\rangle^{\Vill,\calG}_\beta &=&
	\EE^{\ZGF,\calG^*}_{1/\beta}\left[ T^+_{\gamma_{yx} }\right] \,  \geq \, \EE^{\ZGF,\calG^*}_{1/\beta}\left[ T^+_{\gamma_{yx} }
	\,
	\chi_{A_{\gamma_{yx},e}^{{\widetilde q}}} \right]
\end{eqnarray}
where $T_{\gamma_{yx} }^+$ is the function defined in \cref{eq:T}, $\chi_{A} $ is the indicator function of the event $A$ and the last factor is the conditional expectation of $T_{\gamma_{yx} }^+$ conditioned on the event   $A_{\gamma_{yx},e}^{{\widetilde q}}$. The inequality remains true when  $T^+_{\gamma_{yx} }$ is replaced by
$T^-
_{\gamma_{yx} }$.

Thus, to prove \cref{thm:1} it suffices to show that for any ${\widetilde q}$ which depends only on the values of $n$ on the two faces whose boundaries include the edge $e$:
\be  \label{eq:conditional}
\, \EE^{\ZGF,\calG^*}_{1/\beta}\left[ T^+_{\gamma_{yx} }
\,
\chi_{A_{\gamma_{yx},e}^{{\widetilde q}}} \right] \geq \, \PP^{\ZGF,\calG^*}_{1/\beta}\left[ A_{\gamma_{yx},e}^{{\widetilde q}} \right] \,,
\ee
(i.e.,  the conditional expectation of $T$ conditioned on $A$ is greater or equal to $1$).
\normalcolor
\subsection{The Exploration Process}

To study the conditional expectation conditioned on the event $A_{\gamma_{yx},e}^{{\widetilde q}}$, it is of help to employ an exploration process, in which the variables $n(u)$ are revealed along a dynamically defined sequence of steps which by design uncover the $q$-path $\kappa_{xy}^q$ from $x$ to $y$ meeting the conditions of the event $A_{\gamma_{yx},e}^{{\widetilde q}}$ as in \cref{def:A_xy^q}.

The exploration is constructed so that if $A_{\gamma_{yx},e}^{{\widetilde q}}$ fails that the exploration will stop in a bounded number of steps and reveal this fact.   And if the path exists, the exploration will locate it, in a manner which allows a relatively simple description of the conditional distribution of $n$ conditioned on the information revealed by the natural ``stopping time''.

Crucially, the process outputs two sets of faces:  $\calL$ a collection of faces on which $n<q$ (lower), and  $\calH$  a collection of  faces on which $n>q$ (higher), see \cref{fig:schematic of gamma path}, such that $\kappa_{xy}^q$ will be measurable with respect to the values of $n$ on $\calL\cup\calH$.

We now outline the steps of the process, starting with the case of preselected value of $q$:
\begin{enumerate}
	\item Start by revealing the values of $n$ on $\calH_1 := x^+$ and $\calL_1 := x^-$, the pair of faces along the initial edge step $e$. In case $n_{x^-} >q$ or $n_{x^+} < q$, we learned that $A_{\gamma_{yx},e}^{{\widetilde q}}$  fails, and the exploration stops.  Otherwise we move to the next step.
	\item To define the $(k+1)$th step given all $k$ steps before,  expose the faces surrounding $w_k$ in a counter-clockwise fashion, starting from the one to the right of the edge $(w_{k-1},w_k)$, until an is reached which has $n <q$ on its right and $n >q$ on its left.   For reasons explained above, such an edge will be reached.     If the added edge does not terminate on $\gamma_{yx}$, the process is repeated.    Otherwise the exploration stops.
	
	\item From the thus constructed path one determines whether the event $A_{\gamma_{yx},e}^{q}$ occurred or not.  It does iff the path terminates at $y$ and complements $\gamma_{xy}$ into a loop of the relevant orientation, in which case the constructed path is the path $\kappa_{xy}^q$ by which $A_{\gamma_{yx},e}^q$  is defined.
	
\end{enumerate}

It is of relevance to note that in case of success all faces that are revealed in this process have  $n <q$ to the right of the path and $n>q$ on the left.  Their respective unions are the sets $\calL$ and $\calH$ respectively.

Though we shall not use it in the argument,  let us note that, in the standard percolation terminology, the $\calL$ set is connected while the $\calH$ set is $*$-connected.

For any $q \in \Z +1/2$, define $\mathcal{F}_{xy}^{q}$ to be the $\sigma$-algebra generated by the
values of the field $n$ on $\calL \cup \calH$, i.e. the faces exposed by the exploration process for $q$.

The above construction extends simply  to the more general class of exploration processes for paths $\kappa_{xy}^{{\widetilde q}}$ whose $q$ is given by a function ${\widetilde q}$ which depends only on the first two bits of information, i.e.  $n_{x^+}$ and $n_{x^-}$.
Naturally, the corresponding $\sigma$ algebra is denoted $\mathcal{F}_{xy}^{{\widetilde q}}$.

The observations made above suffice to conclude the following statement.

\begin{lem}\label{lem:Exploration}
	For any function ${\widetilde q}:(n_{x^+},n_{x^-})\mapsto\ZZ+\frac{1}{2},$ the event $A_{\gamma_{yx},e}^{{\widetilde q}}$ is $\mathcal{F}_{xy}^{\widetilde q}$-measurable. In addition, if ${\widetilde q} > 0$, the interior of $\boldsymbol{\kappa}_{xy}^{\widetilde q} \cup \gamma_{yx}$ includes $\calH$ and is disjoint from $\calL$. If ${\widetilde q} <0$, then $\calH$ is disjoint from the interior of $\boldsymbol{\kappa}_{xy}^{\widetilde q} \cup \gamma_{yx}$ and $\calL$ is included in it.	
\end{lem}

\subsection{Proof of \cref{thm:1}}

We start with  the simpler statement, proving  \cref{key_bound}  for any preselected value of $q= \Z + 1/2$.

First consider $q > 0$.   Noting that $A_{\gamma_{yx},e}^{{\widetilde q}}\in\mathcal{F}_{xy}^{\widetilde q}$,
from  \cref{lem:Exploration} we conclude $$ \EE\left[ T^+_{\gamma_{yx} }
\,
\chi_{A_{\gamma_{yx},e}^{{\widetilde q}}} \right] = \EE\left[\EE\left[ T^+_{\gamma_{yx} }
\,
\chi_{A_{\gamma_{yx},e}^{{\widetilde q}}} \left.\right|\calF_{xy}^q\right]\right] = \EE\left[\chi_{A_{\gamma_{yx},e}^{{\widetilde q}}}\EE\left[ T^+_{\gamma_{yx} }
| \calF_{xy}^q\right]\right]\,.$$
\normalcolor
For each specified $n$ let
\be \label{def:AB+}
A(n) := \calH(n) \,, \qquad B(n) := \calL(n) \cup \partial\calV^\ast
\ee
and let
$F_A:A\to\ZZ$ and $F_B:B\to\ZZ$ denote the constrained values of $n$ on these two sets (with $\left.F_B\right|_{\partial\calV^\ast} = 0$ under the Dirichlet boundary conditions).

It is now important to note that due to the nature of the exploration process, under the event $A_{xy}^q$ the uncovered values of $n$ satisfy
\begin{align}\label{eq:stiffness conditions}
	\min_{u\in A} F_A(u) -1 \geq \max_{u\in B} F_B(u) \,. \end{align}

Summing over the values which $n$ may assume in the complement of $A\cup B$ we find that for any $n\in A_{xy}^q$:
\begin{align}
	\chi_{A_{\gamma_{yx},e}^{{\widetilde q}}}\EE\left[ T^+_{\gamma_{yx} }
	| \calF_{xy}^q\right]
	 & =
	\frac{ \sum_{n:\calV^\ast\to\ZZ}\exp\left(-\frac12\lambda\norm{\nabla n +\Gamma_{x,y}}^2\right)\chi_{\{\left.n\right|_{A} = F_A\}}\chi_{\{\left.n\right|_{B} = F_B\}}}
	{\sum_{n:\calV^\ast\to\ZZ}\exp\left(-\frac12\lambda\norm{\nabla n}^2\right)\chi_{\{\left.n\right|_{A} = F_A\}}\chi_{\{\left.n\right|_{B} = F_B\}}} \notag \\[1ex]
	& =
	\frac{ \sum_{n:\calV^\ast\to\ZZ}\exp\left(-\frac12\lambda\norm{\nabla n}^2\right)\chi_{\{\left.n\right|_{A} = F_A-1\}}\chi_{\{\left.n\right|_{B} = F_B\}}}
	{\sum_{n:\calV^\ast\to\ZZ}\exp\left(-\frac12\lambda\norm{\nabla n}^2\right)\chi_{\{\left.n\right|_{A} = F_A\}}\chi_{\{\left.n\right|_{B} = F_B\}}}
\end{align}
where for the last ratio we employed the change of variables  in which $n$ is increased over the set enclosed by the loop  $\kappa_{xy}^q \circ \gamma_{yx}$, i.e.
\begin{align}n\mapsto n - \chi_{\Int(\kappa_{xy}^q\circ\gamma_{yx})} \,, \end{align}
which, in the terminology used above, is a gauge transformation which erases the shift operator $T_{\gamma_{yx}}^+ $ while increasing the boundary values along the set $A$.
By the second stipulation of \cref{lem:Exploration} and the fact that the interior must be disjoint from the boundary, this transformation leaves $B$ unaltered.

For a more succinct expression let us denote
\be \label{eq:PartitionFunctionDef} Z_{A,B;F_A,F_B}(n)  := \sum_{n:\calV^\ast\to\ZZ}\exp\left(-\frac12\lambda\norm{\nabla n}^2\right)\chi_{\{\left.n\right|_{A} = F_A\}}\chi_{\{\left.n\right|_{B} = F_B\}}\,,
\ee
where the  $n$ dependence of $A,B$ and $F_A, F_B$ is omitted on the RHS.

In this notation, the above can be summarized by saying that for each $n\in A_{xy}^q$:
\begin{align} \label{Z_ratio}
	\chi_{A_{\gamma_{yx},e}^{{\widetilde q}}}\EE\left[ T^+_{\gamma_{yx} }
	| \calF_{xy}^q\right]
	& =  \frac{Z_{A,B;F_A-\chi_A,F_B}(n)}
	{Z_{A,B;F_A,F_B}(n)}\,.
\end{align}

Under the condition \eqref{eq:stiffness conditions} the partition functions which appear in \cref{Z_ratio} are of a system which is stressed by the imposed boundary conditions, and the stress is higher for the term in the denominator.

This suggests that it may be the case that under the constraint \cref{eq:stiffness conditions},
\begin{align} \frac{Z_{A,B;F_A-\chi_A,F_B} }{Z_{A,B;F_A,F_B}} \geq  1\,.  \label{eq:reduction to stiffness modulus}\end{align}
The suggestive argument should  be taken with a grain of salt, since
it ignores the possibility that the integrality constraints could introduce some unexpected effects in the propagation of strain due to the increase in the  stress  introduced through the alteration of the boundary conditions.

Nevertheless, as we show in the next section, under the stated conditions  \cref{eq:reduction to stiffness modulus} holds true.
We refer to this assertion as an expression of the positivity of the model's stiffness modulus ($\tau$) since, with a stretch of imagination, the quantity
\be  \tau =  - \frac{1}{\lambda |A|} \log \frac{Z_{F_A,F_B}} {Z_{F_A-\chi_A,F_B} }\,
\ee
can be viewed as the pressure needed to be applied along the set $A$, in order to increase the strain there by $1$ (starting from with initial value $F_A-\chi _A$, for which \cref{eq:stiffness conditions} holds).

Combined with  \cref{Z_ratio}, \eqref{eq:reduction to stiffness modulus} allows to  deduce the claimed relation \cref{key_bound} --  for the case $q$ is set at a preselected, positive value.

For $q<0$ the above is to be repeated with the following changes:
\begin{enumerate}[1)]
	\item $T^+_{\gamma_{xy}}$  replaced by $T^-_{\gamma_{xy}}$
	\item in the condition for successful exploration the loop  $\gamma_{yx} \circ \kappa^q_{xy}$ is to be oriented negatively (i.e clockwise) rather than positively
	\item in the definition  of the sets $A(n)$ $B(n)$ the boundary included in the former, i.e. \eqref{def:AB+} replaced by
	\be \label{def:AB-}
	A(n) := \calH(n) \cup \partial\calV^\ast  \,, \qquad B(n) := \calL(n) \,.
	\ee
	\item the expression \eqref{Z_ratio} replaced by
	\begin{align} \label{Z_ratio_-}
		\chi_{A_{\gamma_{yx},e}^{{\widetilde q}}}\EE\left[ T^-_{\gamma_{yx} }
		| \calF_{xy}^q\right]
		& =  \frac{Z_{A,B;F_A,F_B+\chi_B}(n)}
		{Z_{A,B;F_A,F_B}(n)}\,.
	\end{align}
\end{enumerate}

The stiffness condition \eqref{eq:stiffness conditions}  still holds, though now it may be written as
$\min_{u\in A} F_A(u)  \geq \max_{u\in B} F_B(u) +1$.  Under this correspondence, the ratio in  \eqref{Z_ratio_-} still represents the reduction of stress, and by the positivity of the stiffness modulus the relation
\eqref{eq:conditional} holds also in this case.

Finally, to establish the stronger statement that is made in \cref{thm:1} we note that  the above argument readily extends to the case where $q$ instead of being constant is selected as a function of $n_{x^+},n_{x^-}$   (the first variables uncovered by the path exploration process).

Strictly speaking, the above derivation  of  the lower bound on the spin-spin correlation which is expressed in \cref{thm:1}  is  conditioned on the validity of the principle expressed above.  We next turn to its proof.

	\section{Positivity of the $\ZGF$'s stiffness modulus}\label{sec:TM}

The goal of this section is to show that inequality~\eqref{eq:reduction to stiffness modulus} holds, thus completing the proof of~\cref{thm:1}. Specifically, we prove the following proposition

\begin{prop}\label{prop:StiffnessModulus}
Let $A$ and $B$ be disjoint subsets of $\calG$, and $F_A:A \to \mathbb{Z}$ and $F_B: B \to \mathbb{Z}$ be two functions satisfying
\begin{equation}\label{eq:StiffnessBC}
\min_{u \in A} F_A(u) - 1 \geq \max_{u \in B} F_B(u).
\end{equation}
Then, setting $Z_{A,B;F_A-\chi_A,F_B}$ as in\eqref{eq:PartitionFunctionDef},
\[
Z_{A,B;F_A-\chi_A,F_B} \geq Z_{A,B;F_A,F_B}.
\]
\end{prop}

We note that~\cref{prop:StiffnessModulus} can be seen as a special case of the more general result presented in~\cref{subsec:general positivity of stiffness modulus}. Here, we will take advantage of the fact that the $\ZGF$ admits an extension to {\em continuous} functions on the metric graph to present an outline of a simpler, more conceptual argument for the positive of the stiffness modulus. Indeed, one may create such a continuous extension by conditioning on the values of $n$ on all vertices of a finite $\calG$, and then placing independent Brownian bridges with the prescribed endpoint values on each edge. It has been frequently observed that applying such an extension to the real-valued Gaussian Free Field on $\calG$ yields a Markovian process --- i.e. one may determine the distribution on the vertices of $\calH \subset \calG$ by conditioning on the values $n$ on the midpoints of all edges in the edge boundary of $\calH$; see~\cite{lupu2018random} for an example of an application of this property to the study of the Gaussian Free Field.

The aforementioned Markov property also holds for the continuous extension of the $\ZGF$. To see this,  we recall the construction used in the proof of~\cref{thm:deloc on Z^2} (see~\eqref{eq:Interpolation}), where we added a real-valued degree of freedom to the midpoint of each edge, coupled to the $n$ via a Gaussian interaction. We may iterate this process, thus fracturing each into many `vertices' (where $n$ can be defined) without changing the overall elastic energy or the distribution of $n$ on the original vertices of $\calG$. At any finite number of partitions, this field maintains the Markov property. Taking the limit as the number of partitions goes to infinity, the process on each edge converges to a Brownian bridge, and we recover the continuous extension described above.

\begin{proof}[Proof of ~\cref{prop:StiffnessModulus}]
Without loss of generality, let us assume that $\max_{u \in B} F_B(u) \leq 0$ and $ \min_{u \in A} F_A(u) \geq 1$. Let $\bar{\calG}$ be the metric graph, and $\bar{n}$ to be the continuous extension on the $\ZGF$ to $\bar{\calG}$. Define $C^{1/2}$ to be the union of the connected component of $\{s \in \bar{\calG}: \bar{n}(s) \geq 1/2\}$ that intersects $A$, and, similarly, $C^{0}$ to be the union of the connected components of $\{s \in \bar{\calG} :\bar{n}(s) \leq 0\}$ that intersect $B$. By assumption, these sets are disjoint and nonempty.

Given $\bar{n}$, we now define
\[
T(\bar{n}) := \begin{cases} -\bar{n}_s & s \in C^{0} \\ 1-\bar{n}_s & s \in C^{1/2} \\ \bar{n}_s & \text{ otherwise}. \end{cases}
\]
In words, $T$ reflects the portions of the field $\bar{n}$ in $C^0$ about $0$, and the portion of the field in $C^{1/2}$ about $1/2$. If we set $G(\bar{n}) = -T(\bar{n})$, we observe that, for any $\bar{n}$ which is equal to $F_A$ and $F_B$ on $A$ and $B$, respectively, $G(\bar{n})$ is equal to $F_A -1$ and $F_B$ on $A$ and $B$, respectively. Finally, the Markov property of $\bar{n}$(or alternatively, the famous {\em reflection principle} of Brownian motion) imply that $G$ is a measure-preserving map. The existence of such a measure-preserving map which reduces the boundary conditions on $A$ by one implies that
\[
Z_{A,B;F_A-\chi_A,F_B} \geq Z_{A,B;F_A,F_B},
\]
as required.
\end{proof}

\section{Height function fluctuation in relation to spin correlations}

As the last preparatory step towards the proof of \cref{thm:2}, we note the following deterministic relation.  To state it, we recall that $Q_j$ are the four collections of pairs of sites on the boundaries of short translates of the standard quadrants of $\Z^2$ (cf. \cref{fig:PointsOnRays}).  For any $x,y \in Q_j$, we also set $\gamma_{yx}^j$ to be the oriented path on $Q_j$ from $y$ to $x$, starting with the oriented edge $e$ directed towards the interior of the corresponding quadrant.

\begin{lem}\label{lem:Loops}
For any $q \in \ZZ+\frac{1}{2}$,
	\[
	4 \cdot \mathcal{N}^{\mathrm{sgn}(q)}_q(f_0) \, \leq \,  \sum_{j=1}^{4} \sum_{(x,y) \in Q_j} \chi_{A_{\gamma_{yx}^j,e}^{q}}
	\]
	with $\chi_A$  the characteristic function of the event $A$ (of \cref{def:A_xy^q}).
\end{lem}
\begin{proof}
We begin by proving that for $q>0$
\[
\mathcal{N}^{+}_q(f_0) \, \leq \, \sum_{(x,y) \in Q_1} \chi_{A_{\gamma_{yx}^1,e}^{q}}  \,.
\]
	
	Let $\kappa$ be a simple, counter-clockwise oriented $q$-loop surrounding $f_0$. For topological reasons, this loop must include the oriented edge $e$ for some $x$ in the horizontal part of $Q_1$. Let $x$ be the leftmost such edge. Similarly, $\kappa$ must intersect the vertical portion of $Q_1$; we define $y$ to be the first such intersection in counter-clockwise order, beginning at $x$. Then, the portion of the $q$-loop that begins at $x$ and ends at $y$ satisfies the requirements of the event $A_{\gamma_{yx}^1,e}^q$. Thus, every such $q$-loop implies the existence of a pair of points in $Q_1$, and the desired inequality follows.

	For negative values of $q$, we consider clockwise oriented $q$-loops, and define $x$ to be the bottommost point on the vertical portion of $Q_1$ that includes the appropriate oriented edge. We set $y$ to be the first intersection with the horizontal portion, in clockwise order beginning at $x$. The argument then follows. For different values of $j$, the construction is identical, up to rotations.
	\end{proof}

\begin{proof}[Proof of \cref{thm:2}]
	Fix $x$ and $y$ in $\calG$. Using \cref{thm:1}, we learn that for any function ${\widetilde q}(n)$  with values in $\ZZ+\frac{1}{2}$ which depends only on $\{n_{x^+},n_{x^-}\}$
\be  \langle \sigma_x \cdot \sigma_y\rangle^{\Vill,\calG}_\beta \geq \PP^{\ZGF,\calG^*}_{\lambda} \left[A_{\gamma_{yx}^j,e}^{{\widetilde q}}\right] = \EE^{\ZGF,\calG^*}_{\lambda}\left[\PP^{\ZGF,\calG^*}_{\lambda} \left[A_{\gamma_{yx}^j,e}^{{\widetilde q}}\,\big|\,n_{x^+},n_{x^-}\right]\right]\,. \label{eq:special q stiffness modulus result}
\ee
We apply that to the function ${\widetilde q}$ which returns the value of $q$ maximizing the conditional probability of $A_{xy}^q$ to occur, i.e.
$$  {\widetilde q}(n) := \operatorname{argmax}_{q\in \Set{n_{x^-}+\frac12,\dots,n_{x^+}-\frac12}} \left\{\PP^{\ZGF,\calG^*}_{\lambda} \left[A_{\gamma_{yx}^j,e}^{{\widetilde q}}\,\big|\,n_{x^+},n_{x^-}\right] \right\}.$$
 The maximum is attained since the specified values of $n$ leave only
$|n_{x^+}-n_{x^-}|$  options for $q$ at which the probability does not vanish.
	
Recalling  the definitions
	\[
	\mathcal{N}^{+,+}(f_0) = \sum_{q\in\ZZ+\frac{1}{2}:q>0} \calN_q^+(f_0) \,, \qquad  \mathcal{N}^{-,-}(f_0) = \sum_{q\in\ZZ+\frac{1}{2}:q<0} \calN_q^-(f_0)\, ,
	\]
and summing the inequality in \cref{lem:Loops} over all $q \in \Z + 1/2$, we may conclude that

	\begin{align} 4 \cdot  \EE[ \mathcal{N}^{+,+}(f_0) + \mathcal{N}^{-,-}(f_0)] & \leq \sum_{q\in\ZZ+\frac{1}{2}}\sum_{j=1}^{4} \sum_{(x,y) \in Q_j} \PP[A_{\gamma_{yx}^j,e}^{q}] = \sum_{j=1}^{4} \sum_{(x,y) \in Q_j} \EE\left[\sum_{q\in\ZZ+\frac{1}{2}}\chi_{A_{\gamma_{yx}^j,e}^q}\right]. \end{align}
Once the sum over $q$ is placed within the expectation, its convergence becomes clear since
	\begin{align} \EE\left[\sum_{q\in\ZZ+\frac{1}{2}}\chi_{A_{\gamma_{yx}^j,e}^q}\right]
	&\leq\EE\left[|n_{x^+}-n_{x^-}| \chi_{A_{\gamma_{yx}^j,e}^{{\widetilde q}}}\,\right]
	 \\
	&\leq \EE\left[|n_{x^+}-n_{x^-}| ^{1/\ve}\right]^{\ve}\PP\left[A_{\gamma_{yx^j},e}^{{\widetilde q}}\right]^{1-\ve} \tag{H\"older, for any $\ve>0$}\\
	&\leq C_{\ve}\sqrt{\beta} \, \left(\langle \sigma_x \cdot \sigma_y\rangle\right)^{1-\ve} \tag{Using \cref{eq:special q stiffness modulus result} and defining $C_\ve$}
	\,.\end{align}
where the first inequality is deduced through the intermediate conditioning on the values of $n$ at $x^+$ and $x^-$, by noticing that under  this information the number of relevant values of $q$ is at most $|n_{x^+}-n_{x^-}| $, and the  one whose contribution to the sum is maximal is, by definition ${\widetilde q}(n)$.

	The finiteness of $C_\ve$  is a special case of the  Gaussian domination principle  which is recalled below, applied to $v = \delta_{x^+}-\delta_{x^{-}}$.
This, combined with the last two inequalities, yields the claimed \cref{eq:main theorem relation}.
\end{proof}

At the end of the argument we used:
\begin{lem}\label{lem:fluctations of the gradients}
	For the $\ZGF$ on any finite graph, $v\in\ker(-\Delta)^\perp$ and $\ve>0$
\be  \label{eq:epsilon}
\EE\left[|\ip{v}{n}|^{1/\ve}\right]^{\ve} \leq D_\ve \sqrt{ \frac{2\ip{v}{-\Delta^{-1} v}}{\lambda}}
\ee
where  $D_\ve := \left(\Gamma\left(1+\frac{1}{2\ve}\right)\right)^\ve < \infty$.
\end{lem}
\begin{proof}
The  Gaussian domination bound on the generating function \cref{eq:Gaussian domination bound} states
\be  \EE^{\ZGF,\calG^*}_{\lambda}\left[\mathrm{e}^{\langle n,v\rangle}\right]\leq\exp\left(+\frac{1}{2\lambda}\langle v,-\Delta^{-1} v\rangle\right)\qquad (v \perp \ker -\Delta )\,.
\ee
This, through the exponential Chebyshev stratagem yields the estimate
\be
\PP\left[\left\{|\ip{v}{n}|\geq t\right\}\right]\leq 2 \exp\left(-\frac {\lambda}{2}   t^2/ \ip{v}{-\Delta^{-1} v} \right)\,.
\ee
	 The claim then follows through the expectation value's layer-cake representation \,  $ \EE\left[X\right] = \int_0^\infty \PP\left[\Set{X \geq t}\right]\dif{t}$\,.
	\end{proof}

\section{A broader class of models}

	The results presented above for the Villain model and the $\ZGF$ are extended below (in \cref{sec:extension}) to a class of other dual pairs.
Of natural interest is the $XY$ model and its dual, which for reasons  explained below we call the integer-valued Bessel field ($\ZBF$).   However the proofs will be cast in more general terms, covering  $O(2)$ spin models whose dual height functions  can be presented as distributed under annealed Gaussian interactions.   Following are the definitions of the relevant terms.
		
	Throughout this section, $\G=(\calV, \calE)$ is a finite truncation of $\ZZ^2$. In \cref{sec:minorization for general graphs} the construction is extended to more general doubly-periodic planar graphs. 
		
	\subsection{Annealed  Gaussian interactions} \label{sec:annealed_etc}

	\begin{defn}[Annealed Gaussian interactions]\label{def:comp-mon-height-function}	
		A potential function $U: \Z\to \R$  is said to be an \emph{annealed Gaussian interaction} if there exists a finite non-negative Borel measure $\mu_U$ on $\left[0,\infty\right)$ such that for all $q\in \Z$
\begin{align} \exp\left(- U\left(q\right)\right) = \int_{\lambda \in \left[0,\infty\right)}\exp\left(-\frac12\lambda q^2\right)\dif{\mu_U(\lambda)}  \,.\label{eq:annealed_Gaussian}\end{align}
	\end{defn}

To explain the terminology let us note that  inserting \eqref{eq:annealed_Gaussian}  in \eqref{eq:GeneralPotential} right below one finds that the corresponding height function is governed by random Gaussian couplings.  However, while their initial distribution is that of independent random variables with the distribution $\mu_U$,
in the resulting joint distribution of $(n)$ and $(\lambda)$  their probability
is affected by the relative weight of the  partition function at the joint values of
$\{\lambda_{b}\}_{b\in \mathcal E^\ast}$.  In this sense the role played here by the random couplings fits the  statistic-mechanical notion of  an \emph{annealed disorder}.


 \begin{defn}[$\ZUF$ random height function] Given a  function
$U:\Z \to \R$ we denote by  $\ZUF$ the family of integer-valued random height functions associated with the finite graph $\G^\ast$ and the function $U$, with probability measure
	\begin{equation} \label{eq:GeneralPotential}  \PP^{\ZUF,\G^\ast}(n) = \frac{1}{Z^{\ZUF,\G^\ast}} \prod_{\Set{x,y}\in\calE^\ast}\exp\left(- U\left(n_x-n_y\right)\right) \,  \1[\left.n\right|_{\partial \calV^\ast}= 0] \,,
	\end{equation}
where $Z^{\ZUF,\G^\ast}$ is the normalizing factor (the finite graph's partition function, with the Dirichlet boundary conditions).
\end{defn}

To every height function of the $\ZUF$ form is associated a dual $O(2)$ spin model.  The two are related through  Fourier transform, as in the example spelled in \eqref{eq:MBF def} below.
	\begin{defn}[The dual spin model]
Given a $\ZUF$ on a graph $\G^\ast$, its dual spin model is a random spin function $\sigma_x=\mathrm{e}^{i\theta_x}$ defined over the vertex set of $\G$ with the probability distribution
\be \label{dual_O2}   \dif{\PP^{O(2),U,\calG}(\theta)} = \frac{1}{Z^{O(2),U,\calG}}\left(\prod_{\Set{x,y}\in\calE}G_U(\theta_x-\theta_y)\right)\dif{\theta} \,,
\ee
 where $\dif{\theta}$ is the Lebesgue measure on $\left[-\pi,\pi\right)^{|\calV|}$ (free boundary conditions)
and
\be \notag  G_U\left(\vf\right) := \sum_{m\in\ZZ}\exp\left(-\ii m \vf-U\left(m\right)\right)\,.
\ee 		
	\end{defn}

It should be noted that if $U$ is an annealed Gaussian interaction then $G_U$ is  strictly positive. Furthermore, under this correspondence also the spin-spin correlation duality relation  \eqref{eq:duality_correlations} persists. It should be noted that under this duality, $G_U$ corresponds to a "physical" spin model only if it is non-negative.

The class of annealed Gaussian interactions includes the afore-studied integer-restricted Gaussian field $\ZGF$, for which $\mu_U=\delta_\lambda$.   In this special case
under the duality relation  the spin parameter $\beta$  reappears as the inverse of the coupling strength $\lambda$  of the dual $\Z$GF.   In general the correspondence is less explicit.  	
	It is however still true that large spin-$\beta$ implies large fluctuations for the $\ZUF$, in the sense that the  corresponding probability distribution of $n_x$, conditioned on its neighbors,  flattens as $\beta\to\infty$.

\subsection{Divisibility and convexity}

In addition to the above characterization of the interaction  we shall be assuming $U$ is convex. For \cref{thm:gen_deloc on Z^2}, we would require also the following condition
\begin{defn}[divisibility] \label{def:divisibility} An interaction $U$ is said be divisible iff $e^{-U}$ is decomposable   through a continuous or discrete convolution, as either of the two following conditions:
\begin{subequations}	
\begin{align}
\exp\left(-U(a-b)\right) &=  \int_{c\in\RR} \exp(-\widetilde{U}(a-c)) \, \exp(-\widetilde{U}(c-b))\, \dif{c}\qquad(a,b\in\R), 
\label{U_conv_cont_a} \\
\exp\left(-U(a-b)\right) &=   \sum_{c\in\ZZ}\exp(-\widetilde{U}(a-c))  \, \exp(-\widetilde{U}(c-b))\qquad(a,b\in\ZZ)\,,
\label{U_conv_cont}
\end{align}
\end{subequations}
with $\widetilde{U}$   a convex annealed Gaussian interaction.   \\

\end{defn}

\begin{rem}
For  graphs of degree higher than $4$  we shall require a higher degree of divisibility.  For that purpose we  say that the interaction is \emph{r-fold divisible} if
\be
\label{eq:U r-fold conv}
\exp(-U) = [\exp(-\widetilde U) ] ^{*r}
\ee
where $*$ is either the continuum~\cref{U_conv_cont_a} or the discrete convolution operator as in~\cref{U_conv_cont}. \cref{def:divisibility} corresponds then to $r=2$ divisibility.
\end{rem}

\subsection{The $XY$  model and its dual, the $\ZBF$}

The conditions listed above are all applicable to the $XY$ model.
Its  partition function, which was mentioned already in \eqref{Z_XY}, is
\be  \label{Z_XY_2} Z^{XY,\calG}_\beta \equiv \int_{\theta:\calV\to[-\pi,\pi)}\dif{\theta}\prod_{\Set{x,y}\in\calE}\exp\left(\beta\cos\left(\theta_x-\theta_y\right)\right)
\ee
where integration is w.r.t. the Lebesgue measure on $[-\pi,\pi)^{|\calV|}$.

	The model's dual is obtained through the Fourier transform, by which
on each edge the function of $\vf = \theta_x-\theta_y$  is presentable as
		\begin{align} \exp\left(\beta\cos\left(\vf\right)\right)  &=
		\sum_{m\in \Z} \exp\left(-\ii\vf m\right) I_m(\beta)\,, \notag \\
		 I_m(\beta) &:=
		 \frac{1}{2\pi}\int_{\vf=-\pi}^{\pi}\exp\left(\beta\cos(\vf)+\ii\vf m\right)\dif{\vf}\,,\label{eq:MBF def}
\end{align}
$I_m(\beta)$ being the modified Bessel function.
Inserting this in \eqref{Z_XY_2} and integrating over the angles yields the dual representation of the partition function in terms of a sum over height functions $n:\calV^\ast\to\ZZ$ with the Bessel weights
 \begin{align}
  (2\pi)^{-|\mathcal E|}  Z^{\ZBF,\calG^\ast}_\beta =   \sum_{n:
\calV^\ast\to\ZZ:\left.n\right|_{\partial \calV^\ast} = 0}\prod_{\Set{x,y}\in\calE^\ast}I_{n_x-n_y}\left(\beta\right)
 \label{eq:partition function for the ZBF}\end{align}
We shall refer to $n:\calV^\ast\to\ZZ$ with the corresponding normalized probability measure as the integer-valued Bessel field, and denote it by $\ZBF$. \\

The next statement rests on the  observation, which was presented to the third author by A. Raoufi~\cite{Rao},  
that the Gibbs equilibrium measure of the $XY$ model
can be presented as the annealed distribution of a spin model with Villain interactions of random strengths, whose initial distribution is given by an i.i.d. process.  For the completeness of the presentation, we reproduce here Raoufi's argument.

	\begin{lem}\label{lem:ZBF is a mixture of Gaussians} The integer-valued Bessel field (the dual of the $XY$ model) has an annealed Gaussian interaction.
	\end{lem}
	\begin{proof}
	
The proof starts with the observation that the $XY$ Gibbs factor $\exp\left(\beta\cos\left(\vf_v-\vf_u \right)\right) $
for a pair of spins with values  $\mathrm{e}^{\ii\vf_u}$ and $\mathrm{e}^{\ii\vf_v}$,
coincides with the transition amplitude of a Brownian motion in $\R^2$, between a pair of points which happen to be on the unit circle.
Since in this case the transition amplitude depends only on the angular difference, it is natural to  track  the Browning motion
$\left(\gamma(t)\right)_{t\geq0}$  in polar  coordinates, as 
 $\gamma(t) = B(t) \mathrm{e}^{\ii\vf_\gamma(t)}\,. $
The modulus  $B(t) := \norm{\gamma(t)}$ has the distribution of the two-dimensional Bessel process starting at $B(0)=1$ and conditioned on
 $B(1/\beta)=1$.
Conditioned on the process $B(t)$, the accumulated angle $\vf_\gamma(t)$ forms a continuous martingale, whose variance
grows at the rate
\be  \label{random_lambda}
\eta(t) := \int_0^{t}\frac{1}{B(t')^2}\dif{t'}\,.
\ee
(cf. \cite[Corollary 18.7]{Kal02}).

Hence, denoting $\eta(1/\beta) = 1/\lambda$,
\be \label{Bessel_lambda}
\exp\left(\beta\cos\left(\vf\right)\right) =\mathrm{e}^\beta   \int_0^\infty {\sqrt {\frac {\lambda}{2\pi}}}
\sum_{m\in \Z} \exp\left( -\frac {1}{2\lambda} (\vf-2\pi m)^2 \right) \, \mu(d\lambda)
\ee
with $\mu $ the  probability distribution of $\eta(1/\beta)^{-1}$ conditioned on $B(1/\beta) =1$.

The assertion made for the  $\ZBF$ random height function follows from \eqref{Bessel_lambda}  through the
Fourier transform by which the two models are related.
\end{proof}


\begin{lem}
The potential corresponding to the modified Bessel function, i.e.
$$  U(m) = -\log\left(\frac{I_m(\beta)}{I_0(\beta)}\right)\,,  $$
is convex and divisible.
\end{lem}
		\begin{proof}
Convexity of $U$ is established in \cite[(1.9)]{Thir51}).
Divisibility in the sense of \eqref{U_conv_cont} is  valid in this case since
			\begin{align}\label{eq:MBF id} I_{n-m}(\beta_1+\beta_2)=\sum_{l\in\ZZ}I_{n-l}(\beta_1)I_{l-m}(\beta_2)\qquad(n,m\in\ZZ;\beta_1,\beta_2>0)\,.  \end{align}
which follows from the convolution theorem  for the Fourier series and the factorization
 $$ \mathrm{e}^{(\beta_1+\beta_2)\cos(\vf)} = \mathrm{e}^{\beta_1\cos(\vf)}\mathrm{e}^{\beta_2\cos(\vf)}\qquad(\vf\in\RR)\,.$$
		\end{proof}

\subsection{Power-law  interactions}

Let us mention in passing that within the  class of annealed Gaussian interactions are also
\be
 U_\alpha(q) = \lambda |q|^\alpha
\ee
with $\alpha\in(0,2)$.

The proof, based on
Bernstein's theorem on monotone functions, is presented in \cref{App_power_potentials}.
One may note that these interactions are convex only for $\alpha\geq1$, and
satisfy \cref{def:divisibility} only for $\alpha> 1$.

\section{The general version of our main results}  \label{sec:extension}

Using the above terminology, we now extend the  three results which were presented initially in the context of the Villain-- $\ZGF$ dual pair to the more general class of models, including the $XY$--$\ZBF$  pair.
The extensions are presented in the order of their earlier versions.

		\subsection{Depinning in general $\Z$UF} \label{sec:gen_depinning}
		
\begin{thm}[Generalization of \cref{thm:deloc on Z^2}]\label{thm:gen_deloc on Z^2}
Let $\calG^\ast_\infty$ be a doubly-periodic tame planar graph of degree $d$,
with an exhausting sequence of finite subgraphs $\Set{\calG_L^\ast}_L$.
Let   $U$ be a convex annealed Gaussian interaction for which $e^{-U}$ is $r$-fold divisible into convolutions of $e^{-\widetilde U}$,  in the sense of   \cref{eq:U r-fold conv},  with
 $	r := 2\ceil{\log_2(d)}$.
If $\widetilde{U}$ satisfies
 \be \label{gen_deloc_cond}
\exp\left(-\frac{d}{2}\, \left(\widetilde{U}(1) -  \widetilde{U}(0)]\right)\right) \geq \frac{1}{2}\,,
\ee
then the $\ZUF$ depins, in the  sense that
\be\label{deepen ZUF}
\lim_{L\to \infty} \PP^{\ZUF,\calG_L}(|n_x|\le t)= 0
\ee
for each $x\in \mathcal V_\infty$ and $t<\infty$.
\end{thm}

Like  the corresponding statement for \ZGFs/, the results follows from the combination of Lammers's theorem for graphs of maximal degree $3$ and the construction and proof of minorization by degree 3 graphs described in \cref{sec:minorization for general graphs}. The  explicit condition \eqref{gen_deloc_cond},  stated here mainly for the purpose of concreteness, can be improved by employing  uneven decompositions of the coupling, as was done for the bound on $\lambda_c(\ZZ^2)$ stated in \cref{thm:deloc on Z^2}.

The relevant minorization statement which yields \cref{thm:gen_deloc on Z^2} is:

\begin{lem}
\label{thm:depinning for ZUF}
Under the assumptions of \cref{thm:gen_deloc on Z^2}
	there exists a doubly-periodic tame planar \emph{multi} graph  $\mathcal{F}^\ast$ of maximal degree three, with vertex set containing $\calV^\ast_\infty$, and an exhausting sequence $\Set{\calF_L^\ast}_L$ such that for each $L$
\be
\EE^{U,\calG_L^\ast}\left[\left(n_x-n_y\right)^2\right] \geq \EE^{\widetilde{U},\mathcal{F}_L^\ast}\left[\left(n_x-n_y\right)^2\right]\qquad(x,y\in\calV^\ast_\infty)\,.
\ee
\end{lem}

For $\ZZ^2$, the procedure by which  we obtain $\mathcal{F}_L^\ast$  out of $\calG_L^\ast$ follows the process described in \cref{sec:depinning beyond degree 3}.
Its basic step is the reduction by one of the degree of a vertex $x$ though the replacement of two of its incident edges by a single one
which links it to a new site of degree $3$.  The interaction between $x$ and the new site is set to  $\exp(-2\widetilde{U}(\cdot))$, and that of the new sites with its two other neighbors is set to $\exp(-\widetilde{U}(\cdot))$.   The operation is  local, and in a proper iteration of this construction the
doubly-periodicity of the original graph is inherited by the resulting one.
(Unlike in the simpler case of \ZGFs/, we omit here the  improvement which may be attainable through uneven splitting of $U$ in the first step.)  The description of the procedure for  more general graphs is presented in   \cref{sec:minorization for general graphs}.

The essential point here is that
in each step of the construction the expectation value of $(n_x-n_y)^2$ can only decrease.
Except for one ingredient, the proof follows step by step the arguments laid in the proof of the special case presented in \cref{thm:deloc on Z^2}.   The one change which needs to be made is to
adapt the Regev Stephens-Davidowitz monotonicity of the moment generating function in its dependence on  the lattice structure (\cref{sub_latt_monotonicity}), to systems with \emph{annealed}  Gaussian interactions.
The corresponding statement is \cref{lem:MU}, stated and proved in \cref{App:depinning_ZUF}.

We mention a technical point, which yields an additional complication: Lammers's theorem shows that, under the above condition, there cannot exist an infinite-volume Gibbs measure for the height $n$ on the minorizing graph $\mathcal{F}^\ast$. In the context of the $\ZGF$, which satisfies an absolute-value FKG property, the fact that pinning implies the existence of such a measure is standard. We delay the proof of this implication for the $\Z$UF to \cref{sec:Pinning} below.

	\subsection{The level-line lower bound on the spin correlation function}\label{subsec:general positivity of stiffness modulus}
	
	The terms used  in \cref{sec:level_lines} for the  description of the  $\ZGF$ level lines are applicable more generally to the other height functions with an annealed Gaussian interaction. Also the relation between the spin's two-point function and the insertion of a defect line for the height-function \eqref{eq:duality_correlations} remains valid replacing the square with $\calV$.  Continuing with these terms, we have the following  extension of \cref{thm:1}.
	
	\begin{thm}	\label{thm:stochastic geometric bound for mixtures of Gaussians}
Let $U$ be an annealed Gaussian interaction, to which is associated the random height function $n:\calV^\ast\to\ZZ$ distributed as $\PP^{\ZUF,\calG^\ast}$ and its dual $O(2)$ spin field $\theta:\calV\to[-\pi,\pi)$ distributed as $\PP^{O(2),U,\calG}$ (defined by \eqref{dual_O2}).

Then for any preselected path $\gamma_{yx}$ linking a pair of sites $x,y\in \calV$, any oriented edge $e= (x,x')$ which does not intersect it,  and any $q\in \Z+\frac{1}{2}$,
		\be \label{gen_key_bound}
		\E^{O(2),U,\calG}\left[\cos\left(\theta_x-\theta_y\right)\right] \,  \geq \,
		\PP^{\ZUF,\calG^\ast}\left[ A_{\gamma_{yx},e}^q \right]\,,
		\ee
	where $A_{\gamma_{yx},e}^q $ is the event defined in \cref{def:A_xy^q}.
		Furthermore, \cref{gen_key_bound} holds also for $q$ replaced by ${\widetilde q}(n)$ which is  determined by  the values of $n$ on the pair of faces sharing the edge $e$ (i.e. by $\{n(x^+), n(x^-)\}$).
	\end{thm}
	
\begin{proof}

		All the steps in the  derivation of \cref{thm:1}, which this generalizes, apply just as stated there,  provided one also has the suitable extension of the positivity of the stiffness modulus, i.e.  \cref{prop:StiffnessModulus} (which was proven above for the $\ZGF$).
To that end, we restate and verify its extended version.
		
		The requisite extension of \cref{prop:StiffnessModulus} should say that  if $A,B\subseteq\calV^\ast$ are two disjoint sets and $f:A\sqcup B\to\ZZ$ is given such that $$ \min_{x\in A} f(x)-1\geq \max_{x\in B}f(x) $$ and $$ Z^{\ZUF,\calG^\ast}_f \equiv \sum_{n:\calV^\ast\to\ZZ\,:\,\left.n\right|_{A\sqcup B}=f} \quad\prod_{\Set{x,y}\in\calE^\ast}\exp\left(-U\left( n_x-n_y\right)\right) $$ then \begin{align} Z^{\ZUF,\calG^\ast}_{f-\chi_A} \geq Z^{\ZUF,\calG^\ast}_{f}\,.\label{eq:general positivity of modulus} \end{align}
		
		Under the stated assumption on $U$  the Gibbs factor for each edge is presentable as $$ \exp\left(-U\left(n_x-n_y\right)\right) = \int_{\lambda_{xy}\in\left[0,\infty\right)}\exp\left(-\frac12\lambda_{xy}\left(n_x-n_y\right)^2\right)\dif{\mu_U\left(\lambda_{xy}\right)}\,. $$
		
		Plugging this expansion into $Z^{\ZUF,\calG^\ast}_f$ we find \begin{align} Z^{\ZUF,\calG^\ast}_f = \int_{\lambda:\calE^\ast\to\left[0,\infty\right)}Z^{\ZGF,\calG^\ast}_{\lambda,f}\dif{\mu_U^{\otimes \calE^\ast}(\lambda)} \label{eq:partition function with disordered couplings}\end{align} where we used the disordered-$\ZGF$ partition function, with the annealed couplings $\lambda:\calE^\ast\to\left[0,\infty\right)$: \begin{align}\label{eq:ZGF with random couplings partition function} Z^{\ZGF,\calG^\ast}_{\lambda,f}\equiv\sum_{n:\calV^\ast\to\ZZ\,:\,\left.n\right|_{A\sqcup B}=f} \prod_{\Set{x,y}\in\calE^\ast}\exp\left(-\frac12\lambda_{xy}\left(n_x-n_y\right)^2\right)\,. \end{align}
		
		But now, fixing $\lambda$ which is allowed to vary from edge to edge, the argument that was laid out in the proof above of \cref{prop:StiffnessModulus} goes through. We find $$ Z^{\ZGF,\calG^\ast}_{\lambda,f-\chi_A} \geq Z^{\ZGF,\calG^\ast}_{\lambda,f} $$ which implies \cref{eq:general positivity of modulus} via \cref{eq:partition function with disordered couplings}.
	\end{proof}
	
	\subsection{Height-function delocalization implies slow decay of correlations}
	
		The preceding theorem establishes a stochastic-geometric  lower bound on the spin model's two-point function.  The next is a generalization of \cref{thm:2}
		by which we proved that delocalization of the $\ZGF$, rules out exponential decay for the Villain model. As before, our next statement holds true for more general doubly-periodic planar graphs tamely imbedded in $\R^2$, but to keep the notation simple the statement is phrased here in the context of $\ZZ^2$.
		
\begin{thm}[Delocalization implies non-summability]\label{thm:Generalization of thm 2}
For any  annealed Gaussian interaction $U$ the  number of level loops of the random height function $\ZUF$ encircling the origin   bears the following relation with the spin-spin correlations in the dual spin model
 \begin{align}
 \notag
		\sum_{j=1}^4 \sum_{(x,y)\in \widetilde Q_j} \left(\EE^{O(2),U,\calG_L}\left[\cos\left(\theta_x-\theta_y\right)\right]\right)^{1-\ve} \, &\geq \, C_{\ve,U}\,  \,\EE^{U,\calG_L^*}\left[\mathcal{N}^{+,+}(f_0) + \mathcal{N}^{-,-}(f_0)\right]\, \\
	&\geq \, 	C_{\ve,U}\,  \,\EE^{U,\calG_L^*}\left[|n(f_0)|\right]
		\label{eq:non-summable decay for XY}
		\end{align}
with $C_{\ve,U}<\infty$ for any $\ve >0$ (and $\mathcal{N}^{\pm,\pm}$ defined in \eqref{N+-}).
	\end{thm}

	\begin{proof}[Proof summary]
		All steps of the  proof of the corresponding statement for the $\ZGF$ (\cref{thm:2}) go through with the exception of the large gradient bound, \cref{lem:fluctations of the gradients}.   Its suitable generalization to potentials considered here is  given below in \cref{lem:large_grad_bound}.   However, as this is just a minor technicality,  we postpone its derivation to first convey the  implication of \eqref{eq:non-summable decay for XY}.
	\end{proof}

	It should be appreciated that in the collection of  pairs of sites $(x,y)$ summed over in \eqref{eq:non-summable decay for XY} no two pairs are shift equivalent.   Hence:
\begin{cor}
\label{corollary_sum} If in the infinite-volume limit on a doubly-periodic graph the height function $\ZUF$ depins, in the sense that
$\EE^{U,\calG_L^*}\left[|n(f_0)|\right] \to \infty $,  while the spin-spin correlations converge pointwise to a shift invariant limit, then this limit satisfies:
\be \label{cor_decay}
\lim_{L\to \infty} \sum_{u\in \mathcal V} \left(\EE^{O(2),U,\calG_L}\left[\cos\left(\theta_0-\theta_u\right)\right]\right)^{1-\ve}  = \infty
\ee
\end{cor}

The assumption of the limit's translation invariance holds true  for the dual pair of the $\ZBF$ and the XY spin model, and other annealed Gaussian interactions $U$ which are also convex.  In the $\ZBF$-XY  case the  invariance of the limit for the spin model follows by known arguments from either the Ginibre inequalities for the $XY$ model or through the FKG properties of the $\ZBF$ which follow from the convexity of the Bessel interaction.

In addition, as is explained in further detail in \hyperref[sec:Introduction]{the Introduction} and in \cref{sec:dichotomy}, for the $XY$ model the  dichotomy in the decay of correlation functions which we proved for the Villain model is a known feature of the $XY$ model (proven in the works of Lieb~\cite{Lie80} and Rivasseau~\cite{Riv80}, and Aizenman-Simon~\cite{AizSim80B}).   Hence in that case
\eqref{cor_decay} can be replaced by the more explicit bound, in the form it was stated for the Villain model in   \eqref{box_bound}.

The following completes the proof of 	\cref{thm:Generalization of thm 2}.

\begin{lem} \label{lem:large_grad_bound}
	Assume $\calG^\ast$ is planar. If $U$ is convex and annealed Gaussian then for any $\alpha>0$ and any $\Set{x,y}\in\calE^\ast$
	\be\label{eq:finite gradient expectation}  \EE^{\ZUF,\calG^\ast}\left[|n_x-n_y|^\alpha\right] \leq D_{\alpha,U}\,
	\ee
with a $\calG^\ast$-independent constant $D_{\alpha,U}\in(0,\infty)$.
\end{lem}

While the statement is cast in terms of the height function,  the proof given below uses a short cut, enabled by referring to the dual $O(2)$ spin model.   This is the only role planarity, which can actually be avoided,  plays in the proof.  The relevant part of the annealed Gaussian property of $U$ is the positivity of the dual spin Gibbs factor  $G_U$ (of \eqref{dual_O2}). In addition, the relevant part of the convexity assumption is the fact that the growth rate of $U$ at infinity is at least linear.

\begin{proof}
	It is useful to partially perform the dual transformation.  Replacing the sum over $n:\calV^\ast\to\ZZ$ (which in the present context can be viewed as a 2-form) by a sum over
$m_{zw}=n_z-n_w$ (a 1-form on the set of oriented edges), where $zw$ is the oriented edge in $\calE^\ast$ obtained by rotating $uv$ clockwise by 90 degrees, the partition function takes the form
\be
 Z^{\ZUF,\calG^\ast} = \sum_{m:\calE\to\ZZ}\exp\left(-\sum_{b\in\calE}U(m_b)\right)\prod_{u\in\calV}\chi_{\Set{(\dif{}^\ast m)_u=0}}  \ee
with
 \be
(\dif{}^\ast m)_u\equiv\sum_{v\in\calV:\Set{u,v}\in\calE}m_{uv} \qquad(u\in\calV)\,.
\ee

    The fact that $U$ is even, convex and non-constant implies that
    \begin{equation}
      \liminf_{m\to\infty} U(m+1) - U(m) >0.
    \end{equation}
    Let $t_0>0$ be an integer and let $c_U>0$ be such that $U(m+1)-U(m)\ge c_U$ for all integer $m\ge t_0$. The following estimate holds for all $t\ge t_0$ and $f:\Z\to[0,\infty)$:
    \begin{align*}
		Q^+(t) :=& \sum_{m_{uv}>t}\mathrm{e}^{-U(m_{uv})}f(m_{uv})
		= \sum_{m_{uv}=t_0}^\infty\mathrm{e}^{-U(m_{uv}+\lfloor t\rfloor + 1 - t_0)}f(m_{uv}+\lfloor t\rfloor + 1 - t_0)\\
		&\le\mathrm{e}^{-c_U(\lfloor t\rfloor + 1 - t_0)}\sum_{m_{uv}=t_0}^\infty\mathrm{e}^{-U(m_{uv})}f(m_{uv}+\lfloor t\rfloor + 1 - t_0)\\
        &\le\mathrm{e}^{-c_U(\lfloor t\rfloor + 1 - t_0)}\sum_{m_{uv}=-\infty}^\infty\mathrm{e}^{-U(m_{uv})}f(m_{uv}+\lfloor t\rfloor + 1 - t_0)\,.
	\end{align*}
    Analogously,
    \begin{equation*}
		Q^-(t) := \sum_{m_{uv}<-t}\mathrm{e}^{-U(m_{uv})}f(m_{uv})
		\le\mathrm{e}^{-c_U(\lfloor t\rfloor + 1 - t_0)}\sum_{m_{uv}=-\infty}^{\infty}\mathrm{e}^{-U(m_{uv})}f(m_{uv}-(\lfloor t\rfloor + 1 - t_0))\,.
	\end{equation*}	
	The estimates on $Q^\pm(t)$ imply that
	$$ \PP^{\ZUF,\calG^\ast}\left[|n_x-n_y|> t\right] \leq 2\mathrm{e}^{-c_Uq(t)}\frac{\sum_{m:\calE\to\ZZ}\exp\left(-\sum_{b\in\calE}U(m_b)\right)\chi_{\Set{\dif{}^\ast m=q(t)\delta_u-q(t)\delta_v}}}{\sum_{m:\calE\to\ZZ}\exp\left(-\sum_{b\in\calE}U(m_b)\right)\chi_{\Set{\dif{}^\ast m=0}}}\, $$
with $q(t):=\lfloor t\rfloor + 1 - t_0$ and where $uv$ is the dual edge to $xy$.
	
	We rely on the equality
$$ \frac{\sum_{m:\calE\to\ZZ}\exp\left(-\sum_{b\in\calE}U(m_b)\right)\chi_{\Set{\dif{}^\ast m=q\delta_u-q\delta_v}}}{\sum_{m:\calE\to\ZZ}\exp\left(-\sum_{b\in\calE}U(m_b)\right)\chi_{\Set{\dif{}^\ast m=0}}} = \EE^{O(2),U,\calG}\left[\mathrm{e}^{\pm\ii q (\theta_u-\theta_v)}\right]\,, $$
which is proved similarly to the relation~\eqref{eq:duality_correlations}. As the right-hand side is bounded from above by $1$, one finds that $$ \PP^{\ZUF,\calG^\ast}\left[|n_x-n_y|> t\right] \leq 2\mathrm{e}^{-c_Uq(t)} = 2e^{-c_U(\lfloor t\rfloor + 1 - t_0)}\qquad(t\ge t_0)\,. $$
The exponential decay of the right-hand side implies~\eqref{eq:finite gradient expectation}.
\end{proof}

\subsection{Pinning implies the existence of infinite-volume translation invariant measure}\label{sec:Pinning}
We close by proving the technical result mentioned above --- pinning implies the existence of an infinite-volume Gibbs measure. The argument presented bypasses the absolute-value FKG property, and uses annealed RSD sublattice monotonicity (\cref{lem:MU}) in its place.
\begin{prop}
	Let $\calG_{\infty}^\ast$ be a doubly-periodic tame planar graph with an exhausting sequence of finite subgraphs $\Set{\calG_L^\ast}_L$, and $U$ a convex, annealed Gaussian interaction. Assume that, for some vertex $x \in \mathcal{V}_{\infty}$, $\sup_{L} \mathbb{E}^{\ZUF,\calG_L^\ast}[n_x^2] < \infty$.  Then the sequence $\left\{\mathbb{P}^{\ZUF,\calG_L^\ast} \right\}_L$ converges to an infinite-volume Gibbs measure. This limit is independent of the choice of the exhausting sequence $\Set{\calG_L^\ast}_L$, and is invariant under the automorphisms of $\calG_{\infty}^\ast$.
\end{prop}

\noindent {\bf Remark.}  The statement actually holds more broadly, with  $\mathcal{G}^\ast_{\infty}$  allowed to be a general quasi-transitive graph. Planarity enters only through the proof of \cref{lem:large_grad_bound}, for which it can  be avoided -- at the cost of a  somewhat more involved argument.

\begin{proof}
	We may assume that $\Set{\calG_L^\ast}_L$ is an increasing sequence of sets, in the sense of inclusion. \cref{lem:MU} implies that, for any $x \in \calV^\ast_{\infty}$, $\{\mathbb{E}^{\ZUF,\calG_L^\ast}[n_x^2]\}_L$ is an increasing sequence, and hence has a limit (which may be infinite, but is finite for at least one $x$). This limit does not depend on the choice of exhausting sequence: to see this, consider two exhausting sequences $\Set{\calG_L^{1,\ast}}_L$ and $\Set{\calG_L^{2,\ast}}_L$. We construct a new exhausting sequence $\Set{\calG_L^{3,\ast}}_L$ by interlacing the two sequences, setting $\calG_1^{3,\ast} = \calG_1^{1,\ast}$ and choosing the $(L+1)$st element to be the smallest set containing $\calG_L^{3,\ast}$ from the other exhausting sequence. Since $\lim_L\mathbb{E}^{\ZUF,\calG_L^{3,\ast}}[n_x^2]$ exists, it is the same as the subsequential limits along the odd or even elements, which give the limit over $\calG_L^{1,\ast}$ and $\calG_L^{2,\ast}$, respectively.

By translating the exhausting sequences, we deduce that $\lim_{L}\mathbb{E}^{\ZUF,\calG_L^\ast}[n_x^2]$ is the same for all vertices in the same orbit under the action of translations on $\calG^\ast_{\infty}$. Since $\calG^\ast_{\infty}$ is doubly periodic (and thus has finitely many orbits), we conclude that $\Set{\lim_{L}\mathbb{E}^{\ZUF,\calG_L^\ast}[n_x^2]}_{x \in \calG_\infty}$ takes on finitely many values, at least one of which is finite by assumption. Using \cref{eq:finite gradient expectation} with $\alpha =2$ we conclude that $\lim_{L}\mathbb{E}^{\ZUF,\calG_L^\ast}\left[(n_x - n_y)^2\right] < \infty$ for any two neighbors $x$ and $y$. We may iterate this bound finitely many times to conclude that all the variances are finite, i.e. $K := \max_{x\in \calG_\infty} \lim_{L}\mathbb{E}^{\ZUF,\calG_L^\ast}[n_x^2]$ is finite. In particular, the sequence $\left\{\mathbb{P}^{\ZUF,\calG_L^\ast} \right\}_L$ is tight, and has subsequential distributional limits.

	We may repeat the same argument as above, replacing $n_x^2$ with $\exp(a n_x)$ for any $a > 0$, and conclude that $\lim_L \mathbb{E}^{\ZUF,\calG_L^\ast}[\exp(a n_x)]$ exists (though it may be infinite), and is independent of the choice of exhausting sequence. Since $U$ is convex and symmetric about $0$, any finite-volume law of $n_x$ is log-concave (see~\cite[Section 8.2]{Sheff05}). This implies that there exists $C >0$, depending on $\mathbb{E}^{\ZUF, \calG_L^\ast}[n_x^2]$ alone, such that $\mathbb{P}[|n_x| > t] \leq \exp(-Ct)$. From the uniform bound on the variances, we can conclude that there exists $A = A(K)$ such that $\lim_L  \mathbb{E}^{\ZUF,\calG_L^\ast}[\exp(a n_x)] < \infty$ for all $|a| < A$ and all $x  \in \calG^\ast_{\infty}$. By H\"{o}lder's inequality, we bootstrap that bound, and find that, for any finite $S \subset  \calG^\ast_{\infty}$ and $\{a_x\}_{x \in S} \subset [-A',A']$ with $A' := A'(K,|S|)$, $\lim_L  \mathbb{E}^{\ZUF,\calG_L^\ast}[\exp(\sum_{x \in S}  a_x n_x)] < \infty$.

	Let us take a subsequence of $\{\calG_L^\ast\}_L$, and extract a further subsequence $\{\calG_{L_j}^{\ast}\}_j$ so that $\mathbb{P}^{\ZUF,\calG^\ast_{L_j}}$ converges to some $\tilde{\mathbb{P}}$. If we take $|a_x| < A'/2$ for all $x\in S$, then the random variable $\exp(\sum_{x \in S}  a_x n_x)$ is uniformly integrable with respect to the sequence of measures $\left\{\mathbb{P}^{\ZUF,\calG_{L_j}^\ast}\right\}$. Hence,
	\begin{equation}\label{eq:SubSeqExpMoment}
		\lim_j \mathbb{E}^{\ZUF,\calG_{L_j}^\ast} \left[\exp\left(\sum_{x \in S}  a_x n_x\right)\right]  = \tilde{\mathbb{E}}\left[\exp\left(\sum_{x \in S}  a_x n_x\right)\right].
	\end{equation}
	In particular, these limits are finite, which implies that $m_{2k}^{a}$, the $2k$th moment of $\sum_{x \in S}  a_x n_x$ under $\tilde{\mathbb{P}}$, is bounded above by $(2k)! \cdot c^{2k}$ for some finite $c$ (and for all $a_x$'s small enough). Therefore, $\sum_{k} (m_{2k}^{a})^{-1/2k}$ is infinite, and by Carleman's condition, the marginal of $\tilde{\mathbb{P}}$ on $\sum_{x \in S}  a_x n_x$ is the only distribution with the given moment generating function. However, the lefthand limit in~\eqref{eq:SubSeqExpMoment} is independent of   the choice of subsequence $\calG_{L_j}^{\ast}$; therefore, all subsequenetial limits of $\mathbb{P}^{\ZUF,\calG_L^\ast}$ must have the same cylinder set moment generating function. Thus, all subsequential limits are the same, and $\mathbb{P}^{\ZUF,\calG_L^\ast} $ converges to an infinite-volume Gibbs measure, which is independent of the choice of exahusting sequence. Automorphism invariance follows from the independence of exhausting sequence by applying the automorphism to the exhausting sequence.
\end{proof}

\appendix

\section{A Lieb-Rivasseau type inequality for the Villain model}
\label{sec:LR for Villain}

The practical goal of this section is to establish that the Villain model's spin-spin correlation functions  obey the stated  inequality, from which follows the dichotomy that was stated in \cref{lem_dichotomy}.  To get there we present two observations which are of independent interest.  The first is that the Villain model admits a natural extension to the metric graph.  The second is that this extended model can be regarded as the $1D$ scaling limit of $XY$ models.   The desired inequality follows then by continuity from the one which is known to hold for general $XY$ spin systems.

While the focus in the main body of the paper is on models of constant coupling strength, in this appendix we allow the coupling constants  to vary over the edges, but assume that  $J_{u,v} >0$ for all $\{u,v\}\in \mathcal E$.

\subsection{The Villain model's extension to the metric graph}

Given a finite planar graph $\G$, we denote by $\mathcal M(\G)$ the metric graph obtained by metrizing the edges of $\mathcal E(\G)$,  assigning to each length $1$, and by $\widetilde \calE$ the corresponding set of oriented edges.  For each face
$F$ of $\G$, we denote by $\partial F$  the edges  bounding  $F$.

Let $W$ be the `white noise' process on $\widetilde \calE$, that associates additive Gaussian random variables $W([a,b])$ to intervals along the metrized oriented edges, which are: i) antisymmetric under orientation flip, ii) additive in the natural sense, iii) independent for any finite collection of disjoint intervals,  and  iv) of  variance
\be
\E_\beta \big(  W([a,b])^2\big) = \beta J_{u,v} |a-b| \,.
\ee
  with $\{u,v\}$ the edge containing the interval $[a,b]$.

A specified realization $W$ is compatible with a  spin configuration $\sigma = \{\mathrm{e}^{i \theta_x}\}_{x\in \mathcal V}$  iff for any pair of neighboring sites $\{u,v\}\in \mathcal V$
 \be \label{W1}
\mathrm{e}^{i \theta_v} \  \mathrm{e}^{-i \theta_v}  \ = \ \mathrm{e}^{i W([u,v])}
 \ee
A necessary and sufficient condition on $W$ for such compatibility for some $\sigma$ is that
\be \label{consistency}
 \mathrm{e}^{i W[\partial F]} = 1  \qquad \mbox{(for all faces $F$ of $\G$)} \,.
\ee
And under this condition,
  the spin function admits a natural extension to the metric graph through the relation
\be
\sigma(x) = \sigma(x_0) \cdot \mathrm{e}^{i \int_{x_0}^x W(du)} \,.
\ee
where $x$ assume all values along the edges of $\mathcal M(\G)$, and the integral can be evaluated (consistently) along any continuous  path from $x_0$ to $x$.
Thus, the angle function $\theta: \mathcal V \to [0,2\pi)$ can be seen as providing a fibration of the set of spin-consistent $W$, as  defined by the condition \eqref{consistency}.

Since  the event defined by \eqref{consistency}  is of zero probability, it is appropriate to comment on the relevant interpretation of the corresponding conditional distribution.
On a finite graph, the regular probability density of this event  is defined here to be  the  $\varepsilon \to 0$ limit of the probability that the linear discrepancies in \eqref{W1} are all bounded by $\varepsilon$, divided by $\varepsilon$.  Likewise, the corresponding \emph{conditional probability measure} conditioned on \eqref{consistency}  is the  $\varepsilon \to 0$ limit of the similarly mollified event.
And, since we are dealing here with events of continuous probability densities, the above can also be expressed through the insertion of the product of the appropriate $\delta$-functions, regarded as Schwartz distributions.

By the independence and the Gaussian structure of $W$,
for any  finite graph (not necessarily planar) and any configuration of angles $\{\theta_x\}_{x\in \mathcal V}$
\begin{multline}
\left( 2\pi\beta\right)^{|\mathcal E|/2} \,
\E_\beta\left(  \prod_{\{x,y\}\in \mathcal E} \delta \left( \big|\mathrm{e}^{i\theta_v} \mathrm{e}^{-i\theta_u} - \mathrm{e}^{i W([u,v])} \big|  \right) \right)  \, = \\
 =  \,
\prod_{\{u,v\}\in   {\calE}}
 \left(
 \sum_{n_{u,v}\in \Z}  \exp \left[-\frac{\beta J{u,v}}{2} (\theta_v-\theta_u + 2\pi n_{u,v})^2 \right]
  \right) \,.
\end{multline}
where the product is over unoriented edges,  but then each taken with an arbitrarily selected orientation (which does not affect the result).

A quick comparison with the partition function formula \eqref{Villain_Z1} reveals the following.

\begin{thm}
The Gibbs state of the Villain model can be viewed as the restriction to the vertex set $\mathcal V$ of a random continuous function $\theta$, defined over the edges of the metric graph,  whose probability distribution is induced by   the white noise measure described above,  conditioned on the event \eqref{consistency} (in the sense explained above).  Furthermore, conditioned on the spin values at the vertices, the  distribution of $\theta$ along the edges is given by a collection of independent Brownian bridges on the circle $\mathbb{S}^1$, each conditioned to reach the specified values at the edge boundary.
\end{thm}

To summarize:  the values of $\theta$ along edges of the metric graph $\mathcal M(\G)$
provide a continuous extension of the Villain model, with the characteristics  stated above.   Under this extension the distribution of the original spin variables remains unmodified.

\subsection{The metric graph Villain model as the $1D$ scaling limit of $XY$}

In the construction presented next  the above continuum Villain model over $\mathcal M(\G)$  emerges as the  limit of a sequence of $XY$ models formulated over refinements of the original graph.  However, in contrast to the above,  in the $XY$ case the distribution  of the spin variables along the original vertex set keeps  being modified (which is unavoidable, since the two discrete models do not coincide).

The $XY$ model's  refinements  are obtained by splitting each edge of $\G$  through the insertion  of $(N-1)$ new sites $\{x_j\}$, and associating to those  spin variables $\mathrm{e}^{i\theta_{x_j}}$, with the interaction between neighboring spins increased $N$ fold (in comparison to the original lattice model's $J_{u,v}$).

The refined $XY$ model's Gibbs factor for an edge $\{u,v\} \in \mathcal E(\G)$ is thus given by the integral over the intermediate sites:
\be \label{F_edge}
\mathrm{e}^{\mathcal {F}^{XY}_N(\theta_u,\theta_{v})} =
\int_{[0,2\pi]^{N-1} } \exp\left( - \frac{\beta J_{u,v} N}2 \sum_{j=1}^N \left|  \mathrm{e}^{i\theta_{x_j}} - \mathrm{e}^{i\theta_{x_{j-1}}} \right|^2 \right) \prod_{j=1}^{N-1} d \theta_{x_j}
 \ee
 where it is to be understood that $x_0\equiv u$, and $x_N\equiv v$.    From these, the extended $XY$ model's partition function is built as
 \be \label{Z_XY_N}
 Z_{\beta,N}^{XY, \mathcal M(\G)} := \int_{[0,2\pi]^{\mathcal V(\G)}} \prod_{\{u,v\}\in \mathcal E(\G)}\mathrm{e}^{\mathcal {F}^{XY}_N(\theta_u,\theta_{v})} \prod_{q\in \mathcal V(\G)} d \theta_q
 \ee
 The corresponding Gibbs equilibrium states are expectation values of spin functions with respect to the  probability measures consisting of the normalized version of those over which the integrals are taken in the above integrals.    It is convenient for our purpose to extend the definition of $\sigma_x$ to the full metric graph even for $N<\infty$, and we do that through the linear interpolation of $\{\theta_{x_j}\}$.

 Following are the two main results of relevance for our discussion.

 \begin{thm}   \label{thm:XY_to_Vill}
 For any finite graph $\G$ the  correlation functions of the extended $XY$ model  described above converge  to those of the corresponding Villain model on the metric graph $\mathcal M(\G)$:
\be
\lim_{N\to \infty}
\langle \sigma_x \cdot \sigma_y\rangle^{XY,\mathcal M(\G)}_{\beta,N}  = \langle \sigma_x \cdot \sigma_y\rangle^{\Vill,\mathcal M(\G)}_\beta
\ee
where the convergence is uniform in $x,y$, ranging over the edges of   $\mathcal M(\G)$.
\end{thm}

\begin{proof}
To simplify the expressions under consideration, let us note that
 \begin{eqnarray}  \left|  \mathrm{e}^{i\theta_j} - \mathrm{e}^{i\theta_{j-1}} \right|^2  &= & 4   \left(\sin\left(\frac{\theta_j - \theta_{j-1}}2 \right)\right)^2
 \notag \\[1ex]
&= &
 \left( \theta_j - \theta_{j-1}\right)^2  -\frac{1}{12} \left( \theta_j - \theta_{j-1}\right)^4 +
 O\left(\left( \theta_j - \theta_{j-1}\right)^6\right)
 \,,
 \end{eqnarray}
From the leading order it is not difficult to see that for each $\alpha >1$ the integral in \eqref{F_edge} is asymptotically supported on the set of configurations
\be \label{max_estimate}
\mathcal {L}^{\theta_u,\theta_v}_\alpha =
\left\{ \theta \in [-\pi, \pi)^{N-1} : \max_{j=1,...,N}  |\theta_{x_j}-  \theta_{x_{j-1}} |  \leq \alpha \sqrt{\frac{2\ln N}{\beta J_{u,v} N}} \right\} \,.
\ee
 More explicitly, at any $\varepsilon >0$, for large $N$ the contribution to the integral of configurations at which the above is violated is bounded by\footnote{The stated estimate is most natural in case the integral is with free boundary conditions at one of the edge's ends.  In the restricted case, as here, the proof includes an elementary redistribution step.  Still,  this estimate corresponds to the more elementary part of  L\'evy's theorem on the Brownian motion's short time increments.}
 $$
 N\,  \mathrm{e}^{ \beta N  \alpha^2    \ln N /(\beta N) }= \frac{\mathrm{e}^{\varepsilon }}{N^{\alpha^2-1}}
 $$

 Under the condition \eqref{max_estimate}, the last term in \eqref{F_edge} makes only a negligible contribution to the sum in \eqref{F_edge} -- of the order of $ 4\alpha^6\beta^{-1}\,   O\left( \frac{(\ln N)^3}{N}   \right)$.   Thus:
 \begin{multline} \label{F_2}
\mathrm{e}^{\mathcal {F}^{XY}_N(\theta_u,\theta_v)} =  \\
\mathrm{e}^{o(1)}
\int_{\mathcal {L}^{\theta_u,\theta_v}_\alpha } \exp\left\{
- \frac{\beta J_{u,v}  N}2 \sum_{j=1}^{N-1}
\left[ \left( \theta_j - \theta_{j-1}\right)^2  -\frac{1}{12} \left( \theta_j - \theta_{j-1}\right)^4
\right] \right\}
\prod_{j=1}^{N-1} d \theta_{x_j}
 \end{multline}
Rewriting this expression in terms of the rescaled variables
\be
\eta_j :=  \sqrt{\beta J_{u,v}  N} (\theta_j - \theta_{j-1})
\ee
one gets
  \begin{multline} \label{F_3}
\mathrm{e}^{\mathcal {F}^{XY}_N(\theta_u,\theta_v)} = \mathrm{e}^{o(1)}  C_{\beta,N}
\int ...\int_{|\eta_j|\leq \sqrt{2\ln N}}   \sum_{k\in \Z} \delta \left(\frac{1}{\sqrt {N} } \sum_{j=1}^N\eta_j - \sqrt{\beta J_{u,v}}[(\theta_v-\theta_u)  + 2\pi k ]\right)     \times \hfill
\\
\exp\left\{
- \frac{1}2 \sum_{j=1}^{N-1}
\left[ \eta_j^2  -(12\beta J_{u,v}  N)^{-1} \eta_j^4
\right] \right\}
\prod_{j=1}^{N} d \eta_j
 \end{multline}
 with  $C_{\beta,N}$ 
which does not depend on $(\theta_v-\theta_u)$. Any such factor has no effect on the probability distribution of $\theta$.

It should  be noticed that except for the effect of the conditioning on the value of their overall sum,
$\eta_j$ are distributed as i.i.d. variables with finite moments (close to those of the normal distribution).
By the general \emph{local central limit theorem}
(cf. \cite{DobTir77,Bor17})
if follows that  the dependence of resulting Gibbs factor $\mathrm{e}^{\mathcal {F}^{XY}_N(\theta_u,\theta_v)} $  on the angles $\{\theta_u,\theta_v\}$  is asymptotically the same as it would have been had $\frac{1}{\sqrt {N} } \sum_{j=1}^N\eta_j$ been replaced by a normal random variable of the same variance, which for the above   iid variables is
\begin{align}
\rm{Var} &=    \int_{|\eta_j|\leq \sqrt{2\ln N}}   \eta^2
\exp\left\{
- \frac{1}2
\left[ \eta^2  -(12\beta J_{u,v}  N)^{-1} \eta^4
\right] \right\}
 d \eta / \rm{Norm.}
 \notag \\
& = 1+ \frac{1}{\beta J_{u,v} N} + o\left(\frac{1}{N}\right)
\end{align}
More explicitly,
\be
\mathrm{e}^{\mathcal {F}^{XY}_N(\theta_u,\theta_v)} = \mathrm{e}^{o(1)} \widetilde{C}_{\beta,N} \sum_{k\in \Z } \mathrm{e}^{  - \frac{1}{2} \beta J_{u,v} [(\theta_v-\theta_u)  + 2 \pi k ]^2}
\ee

The above implies that  in the limit $N\to \infty$ the distribution of  the extended $XY$ model's spins
along the vertices of $\G$ converges to that of the standard Villain model.

Furthermore, applying Donsker's theorem \cite{Don51,Dud99} on the convergence of the paths of random walks to those of Brownian motion, we conclude that the distribution of the full spin function $\sigma_x$ along the edges of $\mathcal M(\G)$, converges to that of the extended Villain model.  The convergence is in the sense of the distribution of the $L^{\infty}$ distance in the optimal coupling between a pair of random continuous functions.
The stated convergence of the correlation functions readily follows.

Let us stress that in the above analysis the limit $N\to \infty$ is taken first, before the graph's infinite-volume limit.  That is however sufficient for our purpose.   \end{proof}

\noindent {\bf Remarks:}  It should be noted that the proof of convergence given above applies more generally, including to spin systems with arbitrary rotation invariant (bounded)  finite-range interactions.  In this sense the extended Villain model's distribution along an edge emerges as a universal $1D$ ultraviolet limit of models with $O(2)$ symmetric, and asymptotically local, interactions.   
While on a second thought this should not be viewed as surprising, given the theory around the Central Limit Theorem, this observation carries some less immediate implications, such as \cref{cor:Vill}.

\subsection{Useful implications}

As a direct implication of \cref{thm:XY_to_Vill} we get
 the following extended version of \cref{XY_LR}: 

\begin{cor} \label{cor:Vill} For any finite graph $\G$, the ``Lieb-Rivasseau type'' inequality \eqref{XY_LR} holds for both the Villain model on $\G$ and for its corresponding metric graph extension on $\mathcal M(\G)$.
\end{cor}

This statement follows by continuity from the Lieb-Rivasseau proof of  the inequality for the $XY$ model on arbitrary finite graphs.

Another implication of interest, which is obtained by combining \cref{thm:XY_to_Vill} with the known $XY$ Ginibre inequality \cite{Gin70} is:

\begin{cor} \label{cor:InfLim}  The spin-spin two-point function of the Villain model is pointwise monotone in the coupling constants along each edge, and hence in the volume of the system. Consequently, the infinite-volume limit of the two-point function exists.
\end{cor}

\section{Minorization for general graphs}
\label{sec:minorization for general graphs}

The purpose of this section is to present the
minorization construction for graphs of degree greater than $4$.

\begin{proof}[Proof of \cref{thm:depinning for ZUF}]
Since the main idea is sufficiently conveyed by considering the case of homogeneous graphs, i.e. planar doubly-periodic graphs whose unit cell consists of just one vertex, we focus on this case.  The construction described below is easily adapted to graphs of larger, but  finite, periodicity cells.

Splitting each edge through the addition of a mid point, the graph is converted into one which is tiled by copies of a  star graph,  $S_d$, each rooted in a vertex of the original graphs $\G$ (see \cref{fig:triangular graph one vertex}). The idea is to replace each of the star graphs by a tree whose root is at the center of $S_d$, whose leaves are the external vertices of $S_d$, and whose maximal degree is $3$.  Each original edge of $S_d$ may be identified with a path from a leaf of the new tree back to its root.

Eventually all vertices are replaced in this manner and we obtain the new graph which is of maximal degree $3$, see \cref{fig:triangular graph many vertices}.

\begin{figure}[t]
	\begin{subfigure}{.2\textwidth}
	\centering
	\includegraphics[width=\textwidth]{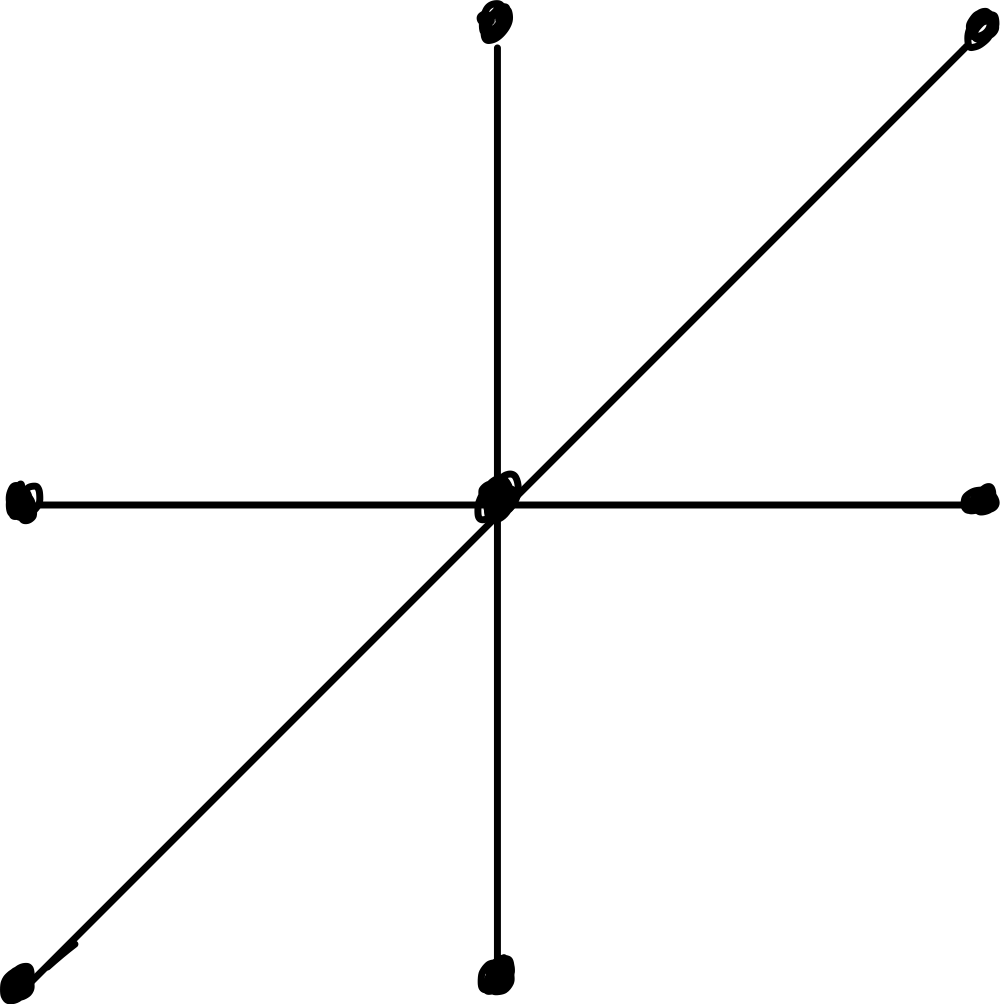}
	\caption{A star within the original triangular graph.}
\end{subfigure}
\qquad
\begin{subfigure}{.2\textwidth}
	\centering
	\includegraphics[width=\textwidth]{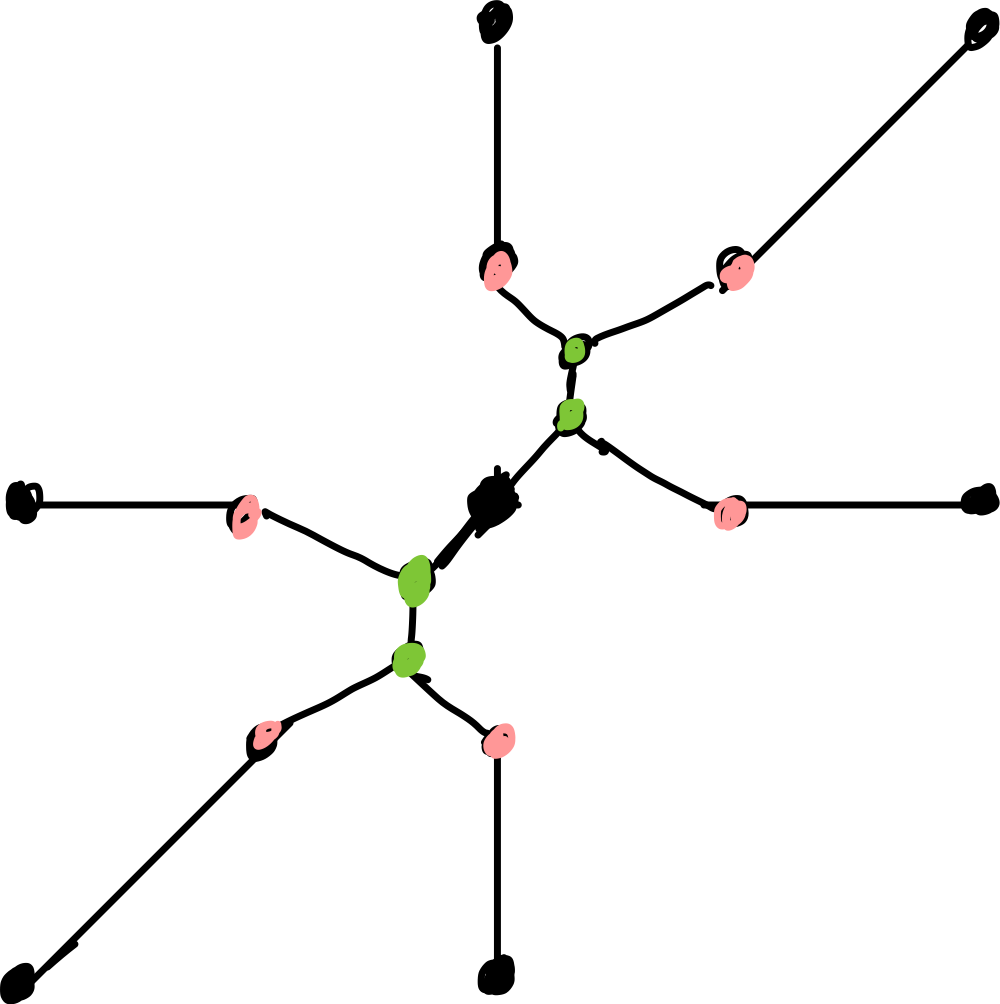}
	\caption{The surgery which happens at each original (black) vertex: green and pink are newly introduced vertices.}
\end{subfigure}
\qquad
\begin{subfigure}{.1\textwidth}
	\centering
	\includegraphics[width=\textwidth]{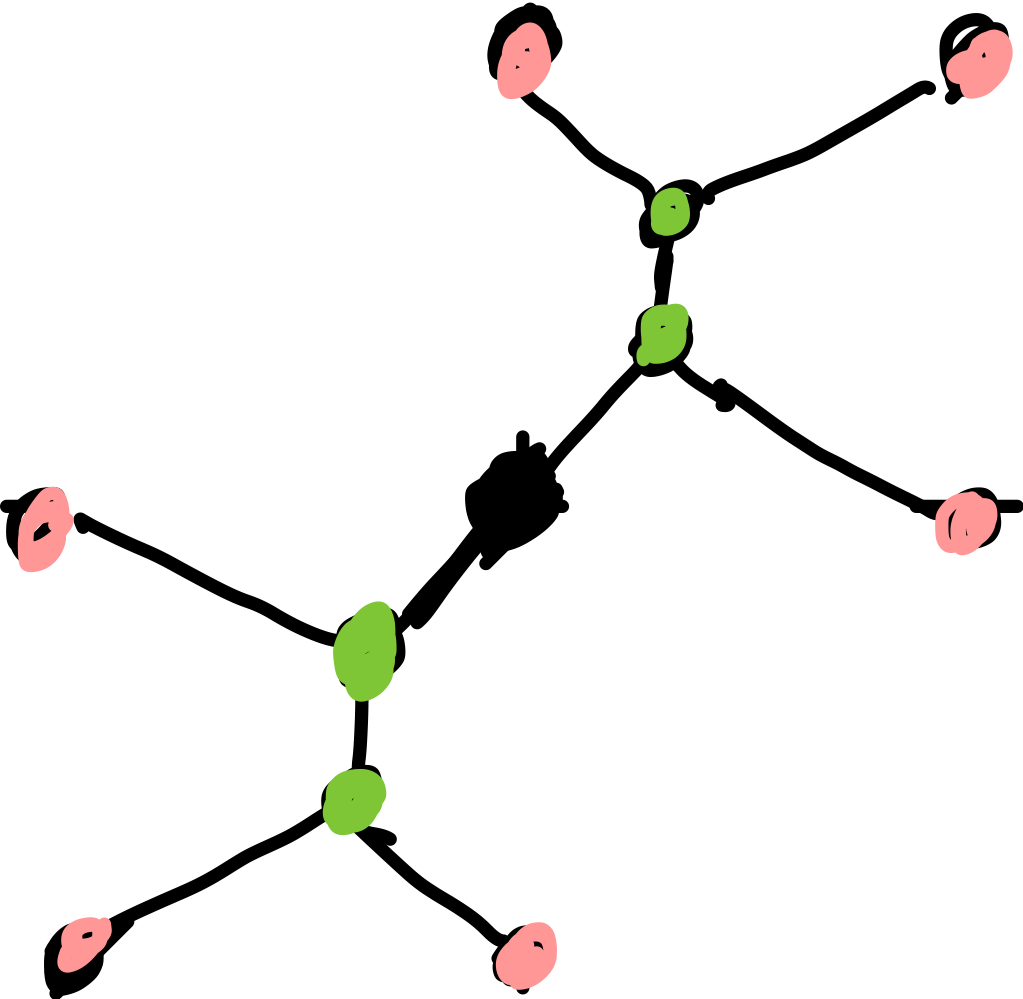}
	\caption{The whole graph is now stitched along the pink vertices.}
\end{subfigure}
	\caption{Example of the procedure at each vertex for the triangular graph.}\label{fig:triangular graph one vertex}
\end{figure}

\begin{figure}[t]
	\centering
	\includegraphics[width=0.5\textwidth]{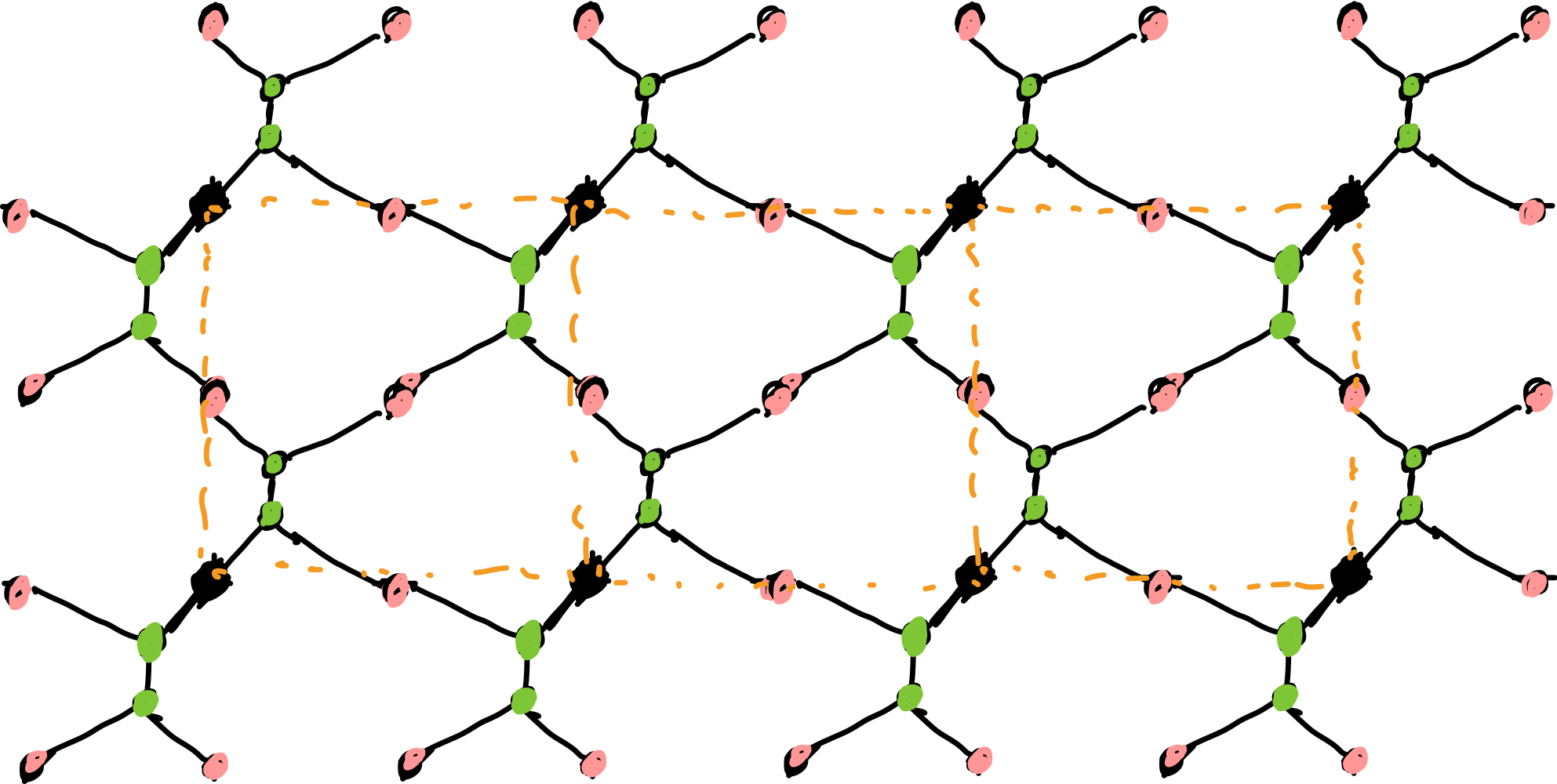}
	\caption{Example of the procedure for a window of the triangular graph. The dashed orange line traces the original unit cell. The maximal degree of this new graph is 3.}\label{fig:triangular graph many vertices}
\end{figure}

	Let $\ell := \ceil{\log_2(d)}$, with $d$  the degree of $\calG$. Divide each edge of the original graph $\calG$ into $r=2\ell$ edges, which is possible since \cref{eq:U r-fold conv} is assumed with multiplicity $2\ell$. Of these $2\ell$ edges, $\ell$ will be used for the surgery of each vertex.
	
	Now, thanks to \cref{lem:MU}, restricting to a sub-lattice will only lower the fluctuations. Hence, restrict to the sub-lattice such that various of the newly-added vertices take the same value so as to replace each star of $d$ edges at each vertex with a tree of $d$ leaves and maximal degree $3$ at each vertex. This is accomplished as follows:
	\begin{itemize}
		\item All new vertices closest to each original vertex (the root of the tree) are divided into two, and we restrict to a sublattice so that each group takes one common value of the field.
		\item The set of next closest vertices are split into two groups, each of which takes the same common value of the field.
		\item This process is repeated $\ell-1$ times (possibly avoiding some restrictions, if $d < 2^\ell$), yielding the desired tree.
	\end{itemize}
	
	Since with each restriction to a sub-lattice the fluctuations go down, we obtain the requisite monotonicity. The new graph clearly has the same periodic structure as the original one, albeit with more complicated unit cell, but with maximal degree three.
\end{proof}
\begin{rem}
	We note that this process did not optimize at all for the best estimate on the critical coupling constant. However, it is clear that since some (new) edges have higher multiplicity than others, one could further optimize the procedure.
\end{rem}

\section{The Regev Stephens-Davidowitz monotonicity theory}\label{sec:RSD monotonicity theory}
	
	For the completeness of presentation, we summarize below the key points of the Regev Stephens-Davidowitz monotonicity theory \cite{RegDav17}, and state their relation to the case of interest here.
	
	A lattice $\calL$ in $\RR^k$ is a set of the form $\calL := L\ZZ^k$ for some $n\in\NN$ and  $L\in\Mat_k(\RR)$.  Given  a  positive-definite matrix $A:\RR^k\to\RR^k$,
a random field $\psi\in\calL$ associated with $\left(A,\calL\right)$ has as its partition function
\be  Z_{A,\calL} := \sum_{\psi\in\calL}\exp(-\frac12\ip{\psi}{A\psi})\,,
\ee
and the  probability distribution of its configurations is given by their  normalized contributions to the sum.

The corresponding moment generating function  is
	\begin{align} \MM_{A,\calL}\left[v\right] := \EE_{A,\calL}\left[\exp\left(\ip{v}{\psi}\right)\right]\qquad(v\in\RR^k)\,.
	\end{align}
	
	The following  result is an important  tool for the derivation of monotonicity estimates.
	\begin{thm}
		For any lattice $\calL$  and a positive matrix $A\geq0$
		\begin{align}\label{Pythagoras_ineq}
			\MM_{A,\calL}\left[u\right]\MM_{A,\calL}\left[v\right] \leq \sqrt{\MM_{A,\calL}\left[u+v\right]\MM_{A,\calL}\left[u-v\right]}\qquad(u,v\in\RR^k)\,.  \end{align}
	\end{thm}
	One may note that in the Gaussian case, where the moment generating function is the exponential of a quadratic form, the relation holds as equality (due to the Pythagorean theorem).
	
	\begin{proof}
		Rewriting the LHS in term of the Cartesian product of lattices $\calL^2 = \calL\oplus\calL$:
\be  \label{double_lattice}
\MM_{A,\calL}\left[u\right]\MM_{A,\calL}\left[v\right] = Z_{A,\calL}^{-2}\sum_{\Psi\in\calL^2}\exp\left(-\frac12\ip{\Psi}{A\oplus A\Psi}+\ip{\begin{bmatrix}	u \\ v
		\end{bmatrix}}{\Psi}\right)\,.
	\ee
		
Under the linear transformation by  $ T = \begin{bmatrix}
			\Id_k & \Id_k \\ \Id_k & -\Id_k
		\end{bmatrix}$,
the lattice $T \calL^2$ is the disjoint union of translates of $(2 \calL)^2$:
\be
 T\calL^2 = \bigsqcup_{w\in\calL/2\calL} (2\calL+w)^2 \,.
\ee

The matrix $\frac{1}{\sqrt 2} T$  is unitary and  $[T,M\oplus M]=0$ for any $n\times n$ matrix $M$ and.  Hence for any $\psi\in\calL^2$
\be-\frac12\ip{ \Psi}{A\oplus A\Psi}+\ip{\begin{bmatrix}
				u \\ v
		\end{bmatrix}}{\Psi} = -\frac14\ip{T\Psi}{A\oplus A T\Psi}+\frac12\ip{\begin{bmatrix}
				u+v \\ u-v
		\end{bmatrix}}{T\Psi} \,.
\ee
Substituting this in \cref{double_lattice} one gets
		\begin{align}\label{eq:interim equation for monotonicity}\begin{split}
				\MM_{A,\calL}\left[u\right]\MM_{A,\calL}\left[v\right] = Z_{A,\calL}^{-2}\sum_{w\in\calL/2\calL}\left(\sum_{\psi\in\calL+w}\exp\left(-\frac14\ip{\psi}{A\psi}+\frac12\ip{u+v}{\psi}\right)\right)\times\\\times\left(\sum_{\psi\in\calL+w}\exp\left(-\frac14\ip{\psi}{A\psi}+\frac12\ip{u-v}{\psi}\right)\right)\,.\end{split} \end{align}
		
In particular, with $v=0$ the last equation yields
		\begin{align}
			\MM_{A,\calL}\left[u\right] = Z_{A,\calL}^{-2}\sum_{w\in\calL/2\calL}\left(\sum_{\psi\in\calL+w}\exp\left(-\frac14\ip{\psi}{A\psi}+\frac12\ip{u}{\psi}\right)\right)^2\,. \label{eq:rewriting of MGF in weird way}\end{align}
		
		Applying now the Cauchy-Schwarz inequality on the RHS of \cref{eq:interim equation for monotonicity} w.r.t. the sum $\sum_{w\in\calL/2\calL}$ and using \cref{eq:rewriting of MGF in weird way} to re-identify $ \MM_{A,\calL}\left[u+v\right],\MM_{A,\calL}\left[u-v\right] $ yields the statement.
	\end{proof}
	
	As was pointed out in \cite{RegDav17}, from the inequality \cref{Pythagoras_ineq} one may deduce
	\emph{sub-lattice monotonicity} of the moment generating function that was stated above as:
	
\noindent{\bf\cref{sub_latt_monotonicity}.}
		 $$ \MM_{A,\calM}\left[v\right]\leq \MM_{A,\calL}\left[v\right] \qquad(v\in\RR^k,A\geq0)\,.$$
for any sub-lattice $\calM\subseteq\calL$.

	\begin{proof} 
		Since $\calM\subseteq\calL$ is a sub-lattice, we have the disjoint decomposition $\calL= \bigsqcup_{w\in\calL/\calM} \calM+w$. Hence \begin{align*}
			Z_{A,\calL}Z_{A,\calM}\MM_{A,\calM}\left[v\right] &=   \left(\sum_{w\in\calL/\calM}\sum_{\psi\in\calM+w}\exp\left(-\frac12\ip{\psi}{A\psi}\right)\right)Z_{A,\calM}\MM_{A,\calM}\left[v\right] \\
			&= \sum_{w\in\calL/\calM}\exp(-\frac12\ip{w}{A w})Z_{A,\calM}^2\MM_{A,\calM}\left[-Aw\right]\MM_{A,\calM}\left[v\right]\,.
		\end{align*}
		Applying now \cref{Pythagoras_ineq} together with $\sqrt{ab}\leq\frac12(a+b)$ on the last two factors of $\MM$ above yields \begin{align*}
			Z_{A,\calL}Z_{A,\calM}\MM_{A,\calM}\left[v\right] &\leq \sum_{w\in\calL/\calM}\exp(-\frac12\ip{w}{A w})Z_{A,\calM}^2\frac12\left(\MM_{A,\calM}\left[-Aw+v\right]+\MM_{A,\calM}\left[-Aw-v\right]\right)\\&= \sum_{w\in\calL/\calM}\exp(-\frac12\ip{w}{A w})Z_{A,\calM}^2\MM_{A,\calM}\left[-Aw+v\right]\tag{$\calM\mapsto-\calM$ symmetry}\\
			&=Z_{A,\calL}Z_{A,\calM}\MM_{A,\calL}\left[v\right] \tag{Use $\calL= \bigsqcup_{w\in\calL/\calM} \calM+w$ again}\,.
		\end{align*}
	\end{proof}
	While it is not used here, it is of interest to note also the following result of \cite{RegDav17} (Theorem 5.1).
		\begin{thm}
		If $\calN,\calM\subseteq\calL$ are sub-lattices then $$ Z_{A,\calM} Z_{A,\calN}  \leq Z_{A,\calL}Z_{A,\calM\cap\calN}\qquad(A\geq0)\,.$$
	\end{thm}
This yields new correlation inequalities for general $\ZGF$ models, for instance:
\be
\PP_\lambda^{\ZGF}\left[\Set{n_x=n_y}\cap\Set{n_u=n_v}\right] \geq \PP_\lambda^{\ZGF}\left[\Set{n_x=n_y}\right] \PP_\lambda^{\ZGF}\left[\Set{n_u=n_v}\right]\,.
\ee

The other property of interest is the operator monotonicity of the moment generating function, i.e.  the monotonicity of
$\MM_{A,\calL}\left[v\right]$ in the matrix $A$ which specifies the interaction.
Its derivation starts with a lower bound on the  Hessian of the moment generating function,
$ \left(\HH \MM\left[u\right]\right)_{ij} \equiv \partial_{u_i}\partial_{u_j} \MM\left[u\right] $.

\begin{lem} For any $\calL$ and $A \geq 0$, as above,  the moment generating function  $\MM[v]=\MM_{A,\calL}\left[v\right]$
satisfies
\be  \label{Hessian_ineq}
\frac{1}{ \MM\left[u\right]}\, \HH \MM\left[u\right] \geq \left.\left(\HH \MM\left[u\right]\right)\right|_{u=0} + \frac{1}{ \MM\left[u\right]^2}\,\nabla \MM\left[u\right]\otimes\left(\nabla \MM\left[u\right]\right)\qquad(u\in\RR^k)\,.   \ee
\end{lem}

In particular,  since
\begin{align} \HH \MM\left[u\right] = \EE\left[\exp\left(\ip{u}{\psi}\right)\psi\otimes\psi\right]  \,; \qquad \nabla \MM\left[u\right] = \EE\left[\exp\left(\ip{u}{\psi}\right)\psi\right]\label{eq:Hessian and gradient of MGF}\end{align} and  $\tr(MN)\geq0$ for any $M,N\geq 0$, \cref{Hessian_ineq} implies
 \begin{align} \EE\left[\exp\left(\ip{u}{\psi}\right)\ip{\psi}{M\psi}\right] \geq \MM\left[u\right]\EE\left[\ip{\psi}{M\psi}\right]\qquad(M\in\Mat_k(\RR):M\geq0,u\in\RR^k)\,.\label{eq:corr ineq for lattice Gaussian fields}\end{align}

	\begin{proof}
		Consider the function $F$ of $u,v\in\RR^k$, $$ F(u,v) := \MM\left[u+v\right]\MM\left[u-v\right]-\MM\left[u\right]^2\MM\left[v\right]^2 \geq 0 $$ where positivity is equivalent to \cref{Pythagoras_ineq}. For any fixed $u$, $F(u,0) = 0$ and hence $v=0$ is a local minimum for each $u$. Hence, $\left(\HH F\right)(u,0)\geq 0$ (the Hessian is w.r.t. $v$).
	A calculation in which it is also used that $\EE\left[\psi\right] = 0$ due to the symmetry $\calL \mapsto -\calL$,
	 shows that $$ \left(\HH F\right)(u,0) = 2\MM\left[u\right]\EE\left[\exp\left(\ip{u}{\psi}\right)\psi\otimes\psi\right]-2\EE\left[\exp\left(\ip{u}{\psi}\right)\psi\right]\otimes\EE\left[\exp\left(\ip{u}{\psi}\right)\psi\right]-2\MM\left[u\right]^2\EE\left[\psi\otimes\psi\right]. $$  The result then follows using \cref{eq:Hessian and gradient of MGF}.
	\end{proof}
This lemma is used in the proof of the matrix-monotonicity which was stated above as:

\noindent{\bf\cref{prop:monotonicity in covariance matrix}.}
		For any lattice $\calL$, and pair of  matrices $A,B$ such that $A\geq B\geq0$
 $$ \MM_{A,\calL}\left[v\right] \leq \MM_{B,\calL}\left[v\right]\qquad(v\in\RR^k)\,.$$
	\begin{proof} 
		Define the homotopy of covariance matrices $$ [0,1]\ni t \mapsto t A + (1-t)B =: A_t\,. $$
		Suffice then to show that $t\mapsto \MM_{A_t,\calL}\left[v\right]$ is monotone decreasing, hence suffice to show $ \partial_t \MM_{A_t,\calL}\left[v\right]~\leq~0$. But $$ 2\partial_t \MM_{A_t,\calL}\left[v\right] = \EE_{A_t,\calL}\left[\ip{\psi}{(A-B)\psi}\right]\MM_{A_t,\calL}\left[v\right]-\EE_{A_t,\calL}\left[\ip{\psi}{(A-B)\psi}\exp(\ip{u}{\psi})\right]\leq 0 $$ the last inequality is \cref{eq:corr ineq for lattice Gaussian fields} with $M=A-B\geq0$.
	\end{proof}
	\begin{cor}\label{cor:monotonicity of ZGF fluctuations w.r.t. coupling}
	For the $\ZGF$ on any graph, the variances $$ \EE_{A,\lambda}^{\ZGF}\left[(n_x-n_y)^2\right] $$ are monotone decreasing in the coupling matrix $A$, and in the overall coupling constant $\lambda$.
	\end{cor}

\section{An annealed version of the RSD monotonicity theory} \label{App:depinning_ZUF}

In the proof of depinning presented in \cref{sec:gen_depinning} (more precisely, proof of \cref{thm:depinning for ZUF})
we invoked  an extension of  the monotonicity principle \cref{sub_latt_monotonicity}
to random height functions with the \emph{annealed} Gaussian interactions such as in  $\PP^{\ZUF,\calG^\ast}$.

Our purpose here is to state and prove this mild extension of the RSD theory.
 For that, in the spirit of \cref{sec:RSD monotonicity theory}, we denote for general lattices $\calL$
	$$Z^{\ZUF,\calG^\ast,\calL} := \sum_{n\in\calL}\prod_{\Set{x,y}\in\calE^\ast}\exp\left(-U\left(n_x-n_y\right)\right)\,.$$
Of interest to us is the case $$\calL \equiv \calB\ZZ^k$$ with $\calB\in\Mat_k(\RR)$ a general matrix, and $k=|\calV^\ast|$. To implement the Dirichlet boundary conditions, $\left.n\right|_{\partial\calV^\ast}=0$, $\calB$ should be chosen such that in its kernel are the field configurations $n$ not obeying these conditions. As for $\calE^\ast$, one may consider it as defining the adjacency structure on $\calL$.
		
		The moment generating function of the corresponding systems is  \begin{align}\MM^{\ZUF,\calG^\ast,\calL}[v] &:= \EE^{\ZUF,\calG^\ast,\calL}\left[\exp\left(\ip{v}{n}\right)\right]\qquad(v\in\RR^k)\,.
		\end{align}
In the case  $U$ is an annealed Gaussian
$$ \MM^{\ZUF,\calG^\ast,\calL}[v] = \frac{1}{Z^{\ZUF,\calG^\ast,\calL}}\int_{\lambda \in (0,\infty)^{\calE^\ast}}\dif{\mu_U^{\otimes \calE^\ast}(\lambda)}Z_{A_\lambda,\calL} \MM_{A_\lambda,\calL}\left[v\right] $$
with
$$ Z_{A_\lambda,\calL} \equiv \sum_{n\in\calL}\exp\left(-\frac12\sum_{\Set{x,y}\in\calE^\ast}\lambda_{xy}\left(n_x-n_y\right)^2\right)=\sum_{n\in\calL}\exp\left(-\frac12\ip{n}{A_\lambda n}_{\RR^k}\right) $$
where $A_\lambda \geq 0$ is minus the discrete Laplacian on the weighted-graph, weighted by $\lambda$.
		
The required extension of \cref{sub_latt_monotonicity} is:
		\begin{lem}\label{lem:MU} Let $\calM\subset\calL$ be a sub-lattice, and $U$ an annealed Gaussian interaction. Then \be  \MM^{\ZUF,\calG^\ast,\calL}[v] \geq \MM^{\ZUF,\calG^\ast,\calM}[v]\qquad(v\in\RR^k)\,.
		\ee
		\end{lem}
		\begin{proof}
At specified couplings, the basic considerations of the RSD monotonicity theory presented in \cref{sec:RSD monotonicity theory} apply to the disordered Laplacian operator $-A_\lambda$ as they do to $-\Delta$; the only necessary ingredient being the positivity $A_\lambda \geq 0$.
			
			Whence, using the monotonicity  \cref{sub_latt_monotonicity},
			\begin{align}\label{eq:naive usage of RSD} Z^{\ZUF,\calG^\ast,\calL}\MM^{\ZUF,\calG^\ast,\calL}[v] &= \int_{\lambda \in (0,\infty)^{\calE^\ast}}\dif{\mu_U^{\otimes \calE^\ast}(\lambda)}Z_{A_\lambda,\calL} \MM_{A_\lambda,\calL}\left[v\right] \notag \\
			& \geq \int_{\lambda \in (0,\infty)^{\calE^\ast}}\dif{\mu_U^{\otimes \calE^\ast}(\lambda)}Z_{A_\lambda,\calM}\frac{Z_{A_\lambda,\calL}}{Z_{A_\lambda,\calM}} \MM_{A_\lambda,\calM}\left[v\right]\,. \end{align}
			
			We define the measure $\widetilde{\PP}^{\ZUF,\calG^\ast,\calM}$ on the disordered couplings $\lambda$ as \begin{align*} \dif{\widetilde{\PP}^{\ZUF,\calG^\ast,\calM}\left(\lambda\right)} \equiv \frac{Z_{A_\lambda,\calM}\prod_{\Set{x,y}\in\calE^\ast}\dif{\mu_U(\lambda_{xy})}}{\int_{\lambda':\calE^\ast\to\left[0,\infty\right)}Z_{A_{\lambda'},\calM}\prod_{\Set{x,y}\in\calE^\ast}\dif{\mu_U(\lambda'_{xy})}}\,. \end{align*} We may rewrite \cref{eq:naive usage of RSD} as  \begin{align*}\MM^{\ZUF,\calG^\ast,\calL}[v]= \widetilde{\EE}^{\ZUF,\calG^\ast,\calL}\left[ \MM_{A_\lambda,\calL}\left[v\right]\right]  &\geq \frac{\widetilde{\EE}^{\ZUF,\calG^\ast,\calM}\left[\frac{Z_{A_\lambda,\calL}}{Z_{A_\lambda,\calM}} \MM_{A_\lambda,\calM}\left[v\right]\right]}{\widetilde{\EE}^{\ZUF,\calG^\ast,\calM}\left[\frac{Z_{A_\lambda,\calL}}{Z_{A_\lambda,\calM}}\right]}\,.
			\end{align*}
			
			Now, the random field $\lambda \in (0,\infty)^{\calE^\ast}$ with probability measure $\widetilde{\PP}^{\ZUF,\calG^\ast,\calM}$ has the positive association property: we consider the FKG lattice structure on the continuous-variable random coupling field $\lambda$ itself, following closely the scheme laid out in \cite[Lemma 2]{CaCha98}. Define on $\lambda$ pointwise monotonicity lattice structure and $$ (\lambda \land \lambda')_b = \min(\Set{\lambda_b,\lambda'_b}) \,;\qquad(\lambda \lor \lambda')_b = \max(\Set{\lambda_b,\lambda'_b})\qquad(b\in \calE^\ast)\,. $$
			
			As in \cite[Lemma 2]{CaCha98}, to establish the FKG property for $\widetilde{\PP}^{\ZUF,\calG^\ast,\calM}$, one verifies that to first order in $\ve_1,\ve_2>0$, and for any two edges $b_1,b_2\in\calE^\ast$, \begin{align}
				\label{eq:CaCha formulation of FKG} Z_{A_\lambda,\calM}Z_{A_{\lambda \lor (\lambda+\ve_1 \delta_{b_1})\lor(\lambda+\ve_2 \delta_{b_2})},\calM} \geq Z_{A_{\lambda \lor (\lambda+\ve_1 \delta_{b_1})},\calM}Z_{A_{\lambda \lor (\lambda+\ve_2 \delta_{b_2})},\calM}\,. \end{align}
			
			To first order in $\ve_1,\ve_2$, denoting $b_1=\Set{x,y},b_2=\Set{u,v}$, both sides of \cref{eq:CaCha formulation of FKG} equal $$ \EE_{A_\lambda,\calM}\left[1-\frac12\ve_1(n_x-n_y)^2-\frac12\ve_2(n_u-n_v)^2\right] $$ and hence the FKG property is established.
			
			Next we turn to the fact that $$\lambda \mapsto \frac{Z_{A_\lambda,\calL}}{Z_{A_\lambda,\calM}} ;\,\qquad\lambda\to\MM_{A_\lambda,\calM}\left[v\right]$$ are monotone decreasing functions. For the moment generating function, this is precisely an application of \cref{prop:monotonicity in covariance matrix}. For the first function, since $\calM\subseteq\calL$, one may rewrite $$ \frac{Z_{A_\lambda,\calL}}{Z_{A_\lambda,\calM}} = \sum_{w\in\calL/\calM}\frac{\sum_{n\in\calM+w}\exp(-\frac{1}{2}\ip{n}{A_\lambda n})}{\sum_{n\in\calM}\exp(-\frac{1}{2}\ip{n}{A_\lambda n})}\,. $$
			
			Now, the fraction $$\lambda\mapsto\frac{\sum_{n\in\calM+w}\exp(-\frac{1}{2}\ip{n}{A_\lambda n})}{\sum_{n\in\calM}\exp(-\frac{1}{2}\ip{n}{A_\lambda n})} =: f_{A_\lambda,\calM}(w) $$ is monotone decreasing thanks to \cite[Prop. 4.2]{RegDav17}; it is a parallel of \cref{prop:monotonicity in covariance matrix}, stated not for the moment generating function but rather for the ratio of shifted partition function by the unshifted partition function, i.e., $f$. The two monotonicity theorems however are equivalent (by completing the square on the moment generating function).
			
			Applying the FKG property, one may conclude the claimed relation
			\begin{align}\MM^{\ZUF,\calG^\ast,\calL}[v] &\geq \widetilde{\EE}^{\ZUF,\calG^\ast,\calM}\left[ \MM_{A_\lambda,\calM}\left[v\right]\right] = \MM^{\ZUF,\calG^\ast,\calM}[v]\,. \end{align}.
		\end{proof}
\section{Gaussian decomposition of power-law interactions}   \label{App_power_potentials}

As a supplement to the discussion of annealed Gaussian interactions (\cref{sec:annealed_etc}), let us add the following example.  Its proof serves to highlight the notion's relation with Bernstein's  theorem.

	\begin{lem}
		The potential  functions    \quad $ U_\alpha(q) = \lambda |q|^\alpha   $
 with $\alpha\in(0,2)$ are presentable as annealed Gaussian interactions (in the sense of \cref{def:comp-mon-height-function}).
	\end{lem}
	\begin{proof}
The decomposability into a superposition of Gaussians is related to
Bernstein's theorem on completely monotone functions.  The latter states that
\be  \exp\left(-U\left(\sqrt{t}\right)\right) = \int_{\lambda=0}^{\infty}\exp\left(-\frac{1}{2}\lambda t\right)\dif{\mu_U(\lambda)}\qquad(t>0)
\ee
 for some probability measure $\mu_U$ on $[0,\infty)$ iff the function
$  F(t) = \exp\left(-U\left(\sqrt{t}\right)\right)   $
 is "totally monotone", i.e.,   is continuous on $\left[0,\infty\right)$, smooth on $\left(0,\infty\right)$, and satisfies \begin{align} (-1)^k F^{(k)}(t) \geq 0 \qquad (k\in\NN_{\geq0},t>0)\,. \label{eq:alternating signs for totally monotonic Bernstein function}\end{align}
		
		To prove the present claim we apply the Bernstein's criterion to the function
				$$ F_\alpha(t) = \exp(-\lambda t^{\alpha/2})\,. $$
Its  regularity properties are clear.  To establish \cref{eq:alternating signs for totally monotonic Bernstein function}, we claim that for each $k\in\NN_{\geq0}$, \begin{align} F_\alpha^{(k)}(t) = F_\alpha(t)(-1)^k\sum_{r=1}^{M_k}a_r t^{p_r} \label{eq:Bernstein induction claim for Lipschitz}\end{align} with $a_r\geq 0$ and $p_r < 0$ for all $r=1,\dots,M_k$,  and $M_k < \infty$.  
		
		The proof is by induction on $k$.   At $k=0$,  \cref{eq:Bernstein induction claim for Lipschitz} is trivially true, and for $k=1$, $$ F_\alpha'(t) = -\lambda\, \frac{\alpha}{2}F_\alpha(t)\, t^{\alpha/2-1} < 0\,. $$
		
		Assuming the induction hypothesis,
		\begin{align*} F_\alpha^{(k+1)}(t) &= \partial F(t)(-1)^k\sum_{r=1}^{M_k}a_r t^{p_r} \\  &= -\lambda\, \frac{\alpha}{2}\, F_\alpha(t)t^{\alpha/2-1}\, (-1)^k\, \sum_{r=1}^{M_k}a_r t^{p_r} + F_\alpha(t)(-1)^k\sum_{r=1}^{M_k}a_r p_r t^{p_r-1}\\
		&=F_\alpha(t)(-1)^{k+1}\sum_{r=1}^{M_k}a_r\left(\lambda\frac{\alpha}{2} \, t^{p_r+\alpha/2-1}+|p_r|t^{p_r-1}\right)
	\end{align*}
		where in the last expression we applied  the condition $p_r < 0$.
		Since also  $\alpha/2-1<0$, the assumed structure extends also to $k+1$.		
	\end{proof}	
	
	\noindent\textbf{Acknowledgements:}
Work on this project was supported in parts by the following grants:
	M.H. was partially supported by the European  Research  Council starting  grant 678520 (LocalOrder).
The research of R.P. is supported by the Israel Science Foundation grant 1971/19 and by the European Research Council starting grant 678520 (LocalOrder).
	J.S.  received support from the Swiss National Science Foundation (grant number P2EZP2\_184228), and the Princeton-Geneva Univ.  collaborative travel funds.
	M.A. thanks Drs. Enrique and Maria Rodriguez Boulan for their warm hospitality at Quogue NY, where some of his work was done, and gratefully acknowledges the support of the Weston Visiting Professorship at The Weizmann Institute of Science (Rehovot, Israel).  
		
\bibliographystyle{plain}
\bibliography{KT}

\end{document}